\documentclass[a4paper]{article}

\let\counterwithin\relax
\usepackage{chngcntr}

\usepackage{amsmath} 
\usepackage{amstext}
\usepackage{amsthm}
\usepackage{amscd} 
\usepackage{amsopn} 
\usepackage{verbatim} 
\usepackage{amssymb}
\usepackage{amsfonts}
\usepackage{fullpage}
\usepackage[bbgreekl]{mathbbol}
\usepackage{tikz}
\usepackage{tikz-cd}
\usetikzlibrary{babel}
\usetikzlibrary{matrix}
\usetikzlibrary{shapes}
\usetikzlibrary{arrows}
\usetikzlibrary{calc,3d}
\usetikzlibrary{decorations,decorations.pathmorphing}
\usetikzlibrary{through}
\tikzset{ext/.style={circle, draw,inner sep=1pt},int/.style={circle,draw,fill,inner sep=1pt},nil/.style={inner sep=1pt}}
\tikzset{exte/.style={circle, draw,inner sep=3pt},inte/.style={circle,draw,fill,inner sep=3pt}}
\tikzset{diagram/.style={matrix of math nodes, row sep=3em, column sep=2.5em, text height=1.5ex, text depth=0.25ex}}
\tikzset{diagram2/.style={matrix of math nodes, row sep=0.5em, column sep=0.5em, text height=1.5ex, text depth=0.25ex}}
\tikzset{every picture/.append style={baseline=-.65ex}}

\usepackage{hyperref}

\newcommand{\ldot}{{\:\raisebox{1.5pt}{\selectfont\text{\circle*{1.5}}}}}
\newcommand{\udot}{{\:\raisebox{4pt}{\selectfont\text{\circle*{1.5}}}}}

\newcommand{\ttt}{\text{-}}

\let\leq\leqslant
\let\geq\geqslant

\newcommand{\NN}{\mathbb N}
\newcommand{\ZZ}{\mathbb Z}

\newcommand{\QQ}{\mathbb Q}

\newcommand{\kk}{\Bbbk}
\newcommand{\bbS}{\mathbb{S}}
\newcommand{\calA}{\mathsf{OS}} 
\newcommand{\calB}{\mathcal{B}}

\newcommand{\calP}{\mathcal{P}}
\newcommand{\calQ}{{\mathcal{Q}}}
\newcommand{\calF}{{\mathcal{F}}}

\newcommand{\calM}{\mathcal M}
\newcommand{\calC}{\mathcal C}

\newcommand{\dual}{\checkmark}

\newcommand{\circles}{
\begin{tikzpicture}[scale=0.2]
\draw (0,0) circle (0.4);
\end{tikzpicture}
}
\newcommand{\odd}{k\circles}  
\newcommand{\even}{k\circles}

\newcommand\Com{\mathsf{Comm}}
\newcommand{\LL}{\mathsf{L}}
\newcommand\Lie{\mathsf{Lie}}
\newcommand{\Pois}{{\mathsf{Pois}}}
\newcommand{\ho}{\mathsf{ho}}
\newcommand{\tkd}{\mathfrak{t}}

\newcommand{\gr}{\mathsf{gr}}

\newcommand{\Top}{\mathcal{T}op}

\newcommand{\orient}{\downarrow}

\newcommand{\extv}[1]{\begin{tikzpicture}
	\node[ext] (v) (0,0) {\small{$#1$}};
	\end{tikzpicture}	}	
	
\newcommand{\GC}{{\mathsf{GC}}}

\newcommand{\hGCor}{{\widehat{\GC}^{\orient}}}
\newcommand{\GCorev}{\GC^{\orient,\even}}
\newcommand{\hGCorev}{{\widehat{\GC}^{\orient,\even}}}

\newcommand{\dGC}{{\mathsf{dGC}}}

\newcommand{\Graphs}{{\mathsf{Graphs}}}

\newcommand{\fGraphs}{{\mathsf{Graphs}}}

\newcommand{\pdu}{{}^*} 

\newcommand{\shufLie}{\mathsf{Shuf}{\text{-}}\mathsf{Lie^k}}
\newcommand{\kgraph}{\mathsf{Grph}^{k}}
\newcommand{\ktree}{\mathsf{Trees}^k}

\newcommand{\ICG}{\mathsf{ICG}}

\newcommand{\ICGS}[1]{{\ICG}^{k\orient}_{#1}(n)}

\newcommand{\Def}{\mathrm{Def}}

\newcommand{\BiDer}{\mathrm{BiDer}}

\DeclareRobustCommand{\rchi}{{\mathpalette\irchi\relax}}
\newcommand{\irchi}[2]{\raisebox{\depth}{$#1\chi$}}

\newcommand{\oneedge}[2]{	
	\begin{tikzpicture}[scale=0.5]
	\node[int] (v) at (0,0) {};
	\node (v1) at (0,0.3) {\small{$#1$}};
	\node[ext] (w) at (1.5,0) {\small{$#2$}};
	\draw[-triangle 60] (v) edge (w);
	\end{tikzpicture}
}

\counterwithin*{equation}{section}

\newtheorem{theorem}[equation]{Theorem}
\newtheorem*{theorem*}{Theorem}
\newtheorem{proposition}[equation]{Proposition}
\newtheorem*{proposition*}{Proposition}

\newtheorem*{statement*}{Statement}
\newtheorem{lemma}[equation]{Lemma}
\newtheorem*{lemma*}{Lemma}
\newtheorem{corollary}[equation]{Corollary}
\newtheorem*{corollary*}{Corollary}
\newtheorem{conjecture}[equation]{Conjecture}
\newtheorem{definition}[equation]{Definition}
\newtheorem*{definition*}{Definition}
\newtheorem{remark}[equation]{Remark}
\newtheorem*{remark*}{Remark}
\newtheorem{example}[equation]{Example}
\newtheorem*{example*}{Example}

\title{Quadratic Algebras arising from Hopf operads generated by a single element}
\author{Anton Khoroshkin\thanks{
		International Laboratory of Representation Theory and Mathematical Physics, 
		National Research University Higher School of Economics, 
		20 Myasnitskaya street, Moscow 101000, Russia \& 	
		Institute for Theoretical and Experimental Physics, Moscow 117259, Russia; } 
	}

\begin{document}
\maketitle
	
\begin{abstract}
The operads of Poisson and Gerstenhaber algebras are generated by a single binary element if we 
consider them as Hopf operads (i.e. as operads in the category of cocommutative coalgebras).
In this note we discuss in details the Hopf operads generated by a single element of arbitrary arity.
We explain why the dual space to the space of $n$-ary operations in this operads are quadratic and Koszul algebras. We give the detailed description of generators, relations and a certain monomial basis in these algebras.

\end{abstract}

\tableofcontents
	
\section{Introduction}

Recall that the homology of the topological operad is an operad in the category of commutative coalgebras which is nowdays called a Hopf operad. Many modern results in deformation theory and homotopy theory of operads deal with Hopf operads (see e.g.\cite{Willwacher_grt},\cite{Fresse}).
On the other hand very few examples of Hopf operads, that has a rigorously simple description is known so far.
One of the main goals of this paper is to give different detailed descriptions of the simplest algebraic Hopf operads with multi-ary generators one can ever encounter. 
Namely, for each $k\in\NN$ and $d\in\ZZ$ we define the Hopf operad called $\Pois^k_d$, such that for $k=1$, $d> 1$ these operads coincides with the homology of the famous little discs operad $E_d$ and $\Pois^1_1$ is equal to the ordinary operad of Poisson algebras.
We define  operads $\Pois^{k}_d$ in Section~\S\ref{sec::Pois::def} as operads (in the category of graded vector spaces) generated by a binary commutative associative product and  the generalized (skew)symmetric Lie bracket that depends on $(k+1)$ arguments, yielding the Leibniz rule~\eqref{eq::Leib::rel} and generalized Jacobi identity~\eqref{eq::Jacobi}.
It turns out, that $\Pois^k_d$ admits a Hopf structure and is generated by a single primitive operation $\nu_k\in \Pois^k_d(k+1)$ of homological degree $(1-d)$ that is supposed to be primitive with respect to the Hopf structure (see Definition~\ref{def::primitive}). This is the reason why we say that  the operads $\Pois^k_d$ are  the simplest Hopf operads.
The cheap statement is that the operads $\Pois^k_d$ are Koszul as operads in the category of graded vector spaces (Proposition~\ref{thm::Pois::Koszul}). However, the proof that for each $n$ the space of $n$-ary operations $\Pois^k_d(n)$ is a quadratic Koszul coalgebra is much more involved (see Theorem~\ref{thm::pairing}). We call  the corresponding algebras  \emph{generalized Orlik-Solomon algebras} and  denote them  by $\calA^k_d(n)$. We present  their definition already in the introduction because of the beauty of this description:
\[\calA^k_d(n):=
\kk\left[
\begin{array}{c}
\omega_S , S\subset \{1,\ldots,n\}, |S|=k+1, \\
\deg(\omega_S)=kd-1, \\
\omega_{\{i_\sigma(1),\ldots,i_{\sigma(k+1)}\}} = (-1)^{d|\sigma|} \omega_{\{i_1,\ldots,i_{k+1}\} }
\end{array} \left|
\begin{array}{c}
\omega_{S} \omega_{T}, \text{ if } |S\cap T| > 1, \\
\omega_{\{i_1 \ldots i_{k+1}\}} \omega_{\{i_{k+1} \ldots i_{2k+1}\}} +  \qquad \qquad \\
\quad+ \omega_{\{i_2 \ldots i_{k+2}\}} \omega_{\{i_{k+2} \ldots i_{2k+1} i_1\}} + \qquad \\ 
\qquad\quad+\ldots+ \omega_{\{i_{k+1} \ldots i_{2k+1}\}} \omega_{\{i_{2k+1} i_1 \ldots i_{k}\}}.
\end{array}
\right.
\right]. 
\]
It is worth mentioning that for the case $k=2$, $d=1$ these algebras is known to coincide with the homology of the real locus of the moduli space of stable rational curves (\cite{Etingof_Rains}), the koszulness of $\calA^k_d(n)$ was proven just recently in~\cite{Khor::Wilw::MonR}.
We also describe the similar generalizations of Drinfeld-Kohno Lie algebras $\tkd_d^k(n)$ that happens to be the Koszul dual to Orlik-Solomon algebras $\calA_d^k(n)$: 
\[
\tkd^{k}_d(n):=
\Lie\left(
\begin{array}{c}
\nu_{I}, \ I\subset [1n], |I|=k+1;\\
\nu_{\sigma(1)\tiny{\ldots}\sigma(k+1)} = (-1)^{d |\sigma|}\nu_{1\tiny{\ldots}k+1}\\
\deg(\nu_I) = 2-kd.
\end{array}
\left|
\begin{array}{c}
[\nu_{I},\nu_{J}], \text{ if } I\cap J =\emptyset, \\ 
{\left[\nu_{I},\sum_{s\in I}\nu_{J\setminus \{t\}\cup \{s\}}\right]. 
	\text{ if } I\cap J =\{t\} }
\end{array}
\right.
\right)
\]
The (co)operadic composition rules for the generalized Orlik-Solomon algebras and Drinfeld-Kohno Lie algebras are natural generalizations of the known classical one (see \eqref{eq::cocompos::OS} and \S\ref{sec::Lie}).
The proof that the coalgebras $\Pois_d^k(n)$ are quadratic uses the  sequence of intermediate equivalent models of these dg-Hopf operads. Let us draw the diagram of all equivalences we are going to prove in the paper:
\[
\begin{tikzcd}
\ho\Pois_d^k(n) \arrow[rrr,"{\S\ref{sec::Graphs::odd}, \text{Theorem}~\ref{thm::H:Graph:1}}"]  \arrow[d,"\S\ref{sec::hoPois}", "H(-)"'] & & & \fGraphs_{d+1}^{k\orient}(n)  & & 
\ICG_{d+1}^{k\orient}(n) \arrow[ll,"C_{\ldot}^{CE}(-)"'] \arrow[d,"{\S\ref{sec::Graphs::odd}, \text{Theorem}~\ref{thm::ICG::Koszul}}", "H(-)"'] \\
\Pois^{k}_{d}(n) \arrow[rrr,leftrightarrow,"{\langle-;-\rangle}","{\S\ref{sec::pairing}}"']& && \calA^{k}_{d}(n) && \tkd^k_d(n) \arrow[ll,"{\text{Corollary}~\ref{thm::dual::OS::KD}}","!"']
\end{tikzcd}
\]

As a byproduct of our theorems we construct a monomial basis of $\calA_d^k(n)$ in Section~\S\ref{sec::OS::basis} and give a description of the generating series (with respect to parameter $n$) of the Hilbert series of dimensions of graded components of the algebras $\calA_d^k(n)$.

The bad news about Hopf operads is that it seems to be a rare situation that the Hopf operad is quadratic (and Koszul) as an operad and all spaces of $n$-ary operations are quadratic (and Koszul) algebras at the same time.
For example, if one considers the Hopf operad generated by a pair of (skew)symmetric primitive elements of the same arity $k$ then the corresponding algebras will be quadratic only for $k=1$, but even for $k=1$ the corresponding algebras does not satisfy the Koszul property as shown in~\cite{DK::Biham}.

\section*{Acknowledgement}
I would like to thank Vladimir Dotsenko, Nikita Markarian, Sergei Merkulov, Dmitri Piontkovski and Thomas Willwacher  for stimulating discussions. I would like to acknowledge Yurii Ivanovich Manin and Bruno Vallette for the correspondence concerning~\cite{Manin::Vallette} which forses me to finish this paper.

My research was carried out within the HSE University Basic Research Program
and funded (jointly) by the Russian Academic Excellence Project '5-100', 
Results of  Section~\S\ref{sec::Graphs::odd}, (in particular, Theorem~\ref{thm::ICG::Koszul}) have been obtained under support of the RSF grant No.19-11-00275. 
I am a Young Russian Mathematics award winner and would like to thank its sponsors and jury.

\section{Recollections and Definitions}
\subsection{Notations for (co)operads}
We refer to the textbooks~\cite{LV,Markl_Shneider} for different standard definitions of operads.
We typically use capital letters for the notations of operads. E.g. we denote by $\Com$ the operad of commutative algebras and by $\Lie$ the operad of Lie algebras. 
The space of $n$-ary operations of the operad $\calP$ is denoted by $\calP(n)$ and in the case we want to mark the inputs by a finite set $I$ of cardinality $n$ we will write $\calP(I)$. 
 
It is often convenient for applications to deal with cooperads, for example a Hopf cooperad is a cooperad in the category of commutative algebras. We will typically use the additional symbol $\pdu$ before the name of the corresponding operad and denote the cooperad $\pdu\calP$ for the dual operad especially in case  $\calP(n)$ are finite-dimensional, for all $n$.
We denote partial cocompositions in the cooperad $\pdu\calP$  by the greek letter $\varphi$:
\begin{equation}
\label{eq::cocompos}
\varphi_{\pdu\calP}^{I,J}: \pdu\calP(I\sqcup J) \rightarrow
\pdu\calP(I\sqcup\{*\}) \otimes \pdu\calP(J)
\end{equation}
and the usual dot-sign is used for the algebra multiplication.

\subsection{Suspension or Homological shift}
We deal a lot with Koszul duality for operads and would like to fix certain conventions.
Let $\calP$ be an algebraic operad. 
We say that a structure of an algebra over a homologically shifted ($k$-suspended) operad $\calP\{k\}$ on a chain complex $V^{\udot}$ is in one-to-one correspondence with the structure of a $\calP$-algebra on a shifted complex $V^{\udot}[k]$. In particular, the homological shift (also called the suspension) increases the homological degree of the space of $n$-ary operations $\calP(n)$ by $(n-1)$ and multiplies with a sign representation $\mathbb{1}_{-}$:
\[
\calP\{1\}(n) := \calP(n)[1-n]\otimes \mathbb{1}_{-}
\]
Following the same ideology the homological shift of a cooperad shifts the degrees of cogenerators in the other direction. In particular, the cooperad $\pdu\Com\{k\}$ is cogenerated by a single element of degree $-k$ which is skew-symmetric for $k$ odd.

In order to preserve the standard conventions suggested by~\cite{Ginz_Kapranov} that predicts the Koszul duality between $\Com$ and $\Lie$ operads we pose the following conventions:
\\
Let $\calP$ be an algebraic operad and $\pdu\calP$ be the corresponding cooperad. Then as a chain complex the cobar construction $\Omega(\pdu\calP)$ is isomorphic to the homological shift of the free operad generated by the shifted symmetric collection $\pdu\calP[-1]:=\cup \pdu\calP(n)[-1]$:
\[
\Omega(\pdu\calP):= \calF(\pdu\calP[-1])\{-1\}
\]
Respectively, we use the following degree conventions for the bar-construction:
\[
\calB(\calP):= \calF^{c}(\calP[1])\{1\}
\]
and we say that a Koszul dual operad $\calP^{!}$ is the homology of the cobar construction $\Omega(\pdu\calP)$. For example, if the Koszul operad $\calP$ is generated by a single symmetric ternary operation of degree $0$ then the operad $\calP^{!}$ is generated by a single skew-symmetric ternary element of degree $1-2=-1$. 

\subsection{Hopf operads}
An operad $\calP$ in a symmetric monoidal category $(\calC,\otimes)$ is called \underline{Hopf} if there exists a coassociative comultiplication $\Delta_{\calP}:\calP \rightarrow \calP\otimes \calP$.
\begin{example}
	An operad $\calP$ in the category $\Top$ of topological spaces is a Hopf operad whose comultiplication is given by the diagonal map.
\end{example}
 The homology (and the chains) of a topological operad provide examples of Hopf operads in (differential) graded vector spaces.
We say that an operad in the category of differential graded commutative coalgebras (dgca for short) is 
\emph{an algebraic Hopf operad}.

Suppose that $\calP$ is an operad in $\Top$ such that the space of $n$-ary operations $\calP(n)$ is a connected CW complex. Then the rational homotopy type $\pi_{\ldot}(\calP_{\QQ})$ form an operad in the category of $\LL_{\infty}$-algebras.  


\begin{lemma}
	\label{lem::pairing}
	The Hopf operad $\calP$ and the Hopf cooperad $\pdu\calP$ are dual to each other iff there exists a collection of nondegenerate pairings
	\begin{equation}
	\label{eq::pairing::def}
	\langle \ttt, \ttt \rangle: \pdu\calP(n) \otimes \calP(n) \to \kk, \ \ \forall n\in\NN
	\end{equation}
	that is compatible with the (co)operadic structure:
	\[\langle f, \alpha\circ_{I,J} \beta  \rangle = 
	\langle (\varphi_{\pdu\calP}^{I,J})^1(f), \alpha \rangle \langle (\varphi_{\pdu\calP}^{I,J})^2(f), \beta \rangle
	\]
	and is compatible with the (co)multiplication (Hopf structure):
	\[\langle f\cdot g, \alpha  \rangle = 
	\langle f, \Delta^{1}(\alpha) \rangle \langle g, \Delta^{2}(\alpha) \rangle.
	\]
\end{lemma}
\begin{proof} Directly follows from the definitions.
\end{proof}	

One of the main viewpoint of this article is to insist the reader that the structure of a Hopf operad is very rigid. Namely, as we will see from the examples, if the Hopf operad $\calP$ is finitely presented as an operad and the algebras $\pdu\calP$ are finitely presented algebras then the pairing\eqref{eq::pairing::def} requires to have a big kernel both on algebra and operadic level, in particular. 
We will try to explain the meaning of this meta-statement for the case of Hopf operads generated by a single element.

\subsection{Primitive elements in Hopf (co)operads}
This section can be omitted because it contains definitions that are not used in the core of the paper.
However, these definitions explains the nature of the examples of the Hopf operads we are working with.

\begin{definition}
	\label{def::Hopf::operad}
	A Hopf operad $\calP$ is called
	\begin{itemize}
		\setlength\itemsep{-0.4em}
		\item 
		{\bf unital} if for all $n$ the coalgebras $\calP(n)$ are counital and the collection of counits $\varepsilon:\calP(n)\twoheadrightarrow \kk$ defines a surjective morphism onto operad of commutative algebras:
		\[ \varepsilon: \calP \twoheadrightarrow \Com\]
		\item 
		{\bf augmented} if for all $n$ the coalgebras $\calP(n)$ admits a coaugmentation $\epsilon:\kk \hookrightarrow \calP(n)$  that assembles a morphism of the operad
		\[\epsilon: \Com \hookrightarrow \calP\]
		such that the composition $\varepsilon\circ \epsilon$ is an identity automorphism  of the commutative operad. We reserve the letter $\mu$ for the image of the commutative associative multiplication in $\calP(2)$.
		\item 
		{\bf {connected}} if there exists a unique $0$-ary operation $\kappa$ providing maps of coalgebras $$\circ_i(\kappa):\calP(n+1)\to \calP(n).$$
		In particular, if $\calP$ is {connected} and augmented then \[\mu\circ\kappa = 1\in \calP(1).\]
		\item
		{\bf graded} if $\calP(n)= \oplus_{i\in\ZZ} \calP(n)_{i}$ is a graded commutative coalgebra such that the operadic compositions preserves grading. 
	\end{itemize}	
\end{definition}
\begin{definition}
	An element $\gamma\in\calP(n)$ of a connected Hopf operad $\calP$ is called 
	\begin{itemize}
		\item 	{\bf nilpotent} if $\gamma\circ_i \kappa =0$ for all $i=1$,\ldots,$n$.
		\item {\bf (skew)symmetric} if $\gamma(x_{\sigma(1)},\ldots,x_{\sigma(n)}) = \pm \gamma(x_1,\ldots,x_n)$.
	\end{itemize}
	
\end{definition}

\begin{definition}
	\label{def::primitive}
	\begin{itemize}
		\item
		An element $\gamma\in\calP(n)$ of a Hopf augmented operad $\calP$ is called {\bf primitive} iff
		\[ \Delta_{\calP}(\gamma) = \underbrace{\mu\circ\mu\circ\ldots\circ \mu}_{n} \otimes \gamma + 
		\gamma \otimes \underbrace{\mu\circ\mu\circ\ldots\circ \mu}_{n}.
		\]
		In other words, $\gamma$ is a primitive element of the coalgebra $\calP(n)$, since the itereted composition of commutative multiplication defines a counit in $\calP(n)$.
		\item
		A collection of primitive elements $\{\gamma_s|s\in S\}$ of a Hopf augmented operad $\calP$ is called {\bf primitive generators} of $\calP$ if the set $\{\mu_2\}\sqcup\{\gamma_s|s\in S\}$ generates the operad $\calP$.
		\item
		A Hopf operad is called {\bf generated by primitives} iff there exists a collection of primitive generators which is minimal as a set of generators.
	\end{itemize}
\end{definition}

\begin{remark}
	Note that when we say that a {connected} operad $\calP$ is generated by the set $S$ we mean that $\calP(0)=\kk\kappa$, $\calP(1)\ni Id$ and $\oplus_{n\geq1}\calP(n)$ is generated by $S$.
	In particular, if $\calP(1)=\kk Id$, one can say that the element $\gamma\in\calP(n)$ of arity $n\geq 2$ is indecomposable if it may not be presented as a sum of compositions of elements of lower arities that are greater than $1$. 
\end{remark}

The key motivation for Definition~\ref{def::Hopf::operad} is the following 
\begin{proposition}
	Let $\calM$ be an operad in the category of topological spaces  then
	\begin{itemize}
		\setlength{\itemsep}{-0.3em}
		\item $H_{\ldot}(\calM;\kk)$ is a graded Hopf operad and respectively $H^{\udot}(\calM;\kk)$ is the corresponding dual graded Hopf cooperad. 
		\item the map $M\to \text{ point}$ makes an operad $H_{\ldot}(\calM;\kk)$ to be unital.
		\item 	If, moreover, for all $n\geq 1$ the spaces of $n$-ary operations $\calM(n)$ are connected then the Hopf operad $H_{\ldot}(\calM,\kk)$ is augmented with $\mu=H_{0}(\calM(2))$.
		\item 
		The operad $H_{\ldot}(\calM,\kk)$ is connected if $\calM(0)$ is a contractible space.
	\end{itemize}
\end{proposition}

\begin{remark}
	There exists other (co)monoidal functors from the category of topological spaces to the category of commutative (co)algebras that produce other examples of connected, graded, unital Hopf operads. However, the augmentation is not given by default even if the spaces of operations of any given arity is connected. One has to deal with pointed spaces.  
\end{remark}

\begin{example}
	The cohomology of the little cubes operad $E_d$ is a unital connected Hopf operad with a unique primitive generator $\nu_2\in H^{d-1}(E_d(2);\QQ) = H^{d-1}(S^{d-1};\QQ) $.
\end{example}

\begin{conjecture}
	If the Hopf operad $\calP$ is generated by primitive nilpotent elements which generates a (Koszul) suboperad $\calQ$ then there exists a ditributive law $\lambda:\calQ\circ\Com \to \Com\circ\calQ$ that implies the isomorphism of symmetric collections $\calP \simeq \Com\circ\calQ$.
\end{conjecture}

\subsection{Hopf operads and Operads in Lie algebras}
The category of Lie algebras admits a simple monoidal structure given by direct sum. 
Consequently, the composition rules in the operad $\mathfrak{g}(n)$ in the category of Lie algebras are given by partial compositions:
\[
\circ_{i} : \mathfrak{g}(m) \oplus \mathfrak{g}(n)  \rightarrow \mathfrak{g}(m+n-1)
\]
Thus, for the Lie algebras there is a nontrivial composition with $0$ element and the partial composition is a direct sum of two maps from $\mathfrak{g}(m)$ and from $\mathfrak{g}(n)$ correspondingly:
\begin{align*}
\circ_{i} \quad &
:=    \text{"}\circ_{i}^{1}:= \circ_{i}\mid_{\mathfrak{g}(m)} \text{"} +  \text{"}\circ_{i}^{2}:= \circ_{i}\mid_{\mathfrak{g}(n)} \text{"} \\
\begin{tikzpicture}[scale=0.8]
\node[ext] (v) at (0,-1) {\tiny{$\mathfrak{g}(m)$}};
\node[ext] (w) at (0.5,0)  {\tiny{$\mathfrak{g}(n)$}};
\coordinate (s) at (0,-1.8);
\coordinate (v0) at (-1.3,0);
\coordinate (v1) at (-0.7,0);
\coordinate (v2) at (1.5,0);
\coordinate (u0) at (-0.3,1);
\coordinate (u1) at (-0.7,1);
\coordinate (u2) at (1.3,1);
\node (s2) at (0.5,1) {\tiny{$\cdots$}};
\draw (v) edge (s) edge (v0)   edge (v1) edge (v2) edge (w);
\draw (w) edge (u0)   edge (u1) edge (u2);
\end{tikzpicture}
& = 
\begin{tikzpicture}[scale=0.8]
\node[ext] (v) at (0,-1) {\tiny{$\mathfrak{g}(m)$}};
\node[ext] (w) at (0.5,0)  {\small{$0$}};
\coordinate (s) at (0,-1.8);
\coordinate (v0) at (-1.3,0);
\coordinate (v1) at (-0.7,0);
\coordinate (v2) at (1.5,0);
\coordinate (u0) at (-0.3,1);
\coordinate (u1) at (-0.7,1);
\coordinate (u2) at (1.3,1);
\node (s2) at (0.5,1) {\tiny{$\cdots$}};
\draw (v) edge (s) edge (v0)   edge (v1) edge (v2) edge (w);
\draw (w) edge (u0)   edge (u1) edge (u2);
\end{tikzpicture}
+ 
\begin{tikzpicture}[scale=0.8]
\node[ext] (v) at (0,-1) {\small{$0$}};
\node[ext] (w) at (0.5,0)  {\tiny{$\mathfrak{g}(n)$}};
\coordinate (s) at (0,-1.8);
\coordinate (v0) at (-1.3,0);
\coordinate (v1) at (-0.7,0);
\coordinate (v2) at (1.5,0);
\coordinate (u0) at (-0.3,1);
\coordinate (u1) at (-0.7,1);
\coordinate (u2) at (1.3,1);
\node (s2) at (0.5,1) {\tiny{$\cdots$}};
\draw (v) edge (s) edge (v0)   edge (v1) edge (v2) edge (w);
\draw (w) edge (u0)   edge (u1) edge (u2);
\end{tikzpicture}
\end{align*}

Note that the Chevalley-Eilenberg complex (as well as its homology) are  monoidal functors from the category of Lie algebras to the category of cocommutative coalgebras:
\[C^{CE}_{\ldot}(\mathfrak{g}\oplus \mathfrak{g}') \simeq C^{CE}_{\ldot}(\mathfrak{g})\otimes C^{CE}_{\ldot}(\mathfrak{g})\]
Consequently, if $\cup\mathfrak{g}(n)$ is an operad in Lie algebras then the Chevalley-Eilenberg complexes $\cup C^{CE}_{\ldot}(\mathfrak{g}(n))$ and $\cup H^{CE}_{\ldot}(\mathfrak{g}(n))$ are Hopf (dg)operads.

\begin{remark}
	If the operad $\mathfrak{g}$ in $\Lie$ has no operations in arity $1$ then the corresponding Hopf operads $C^{CE}_{\ldot}(\mathfrak{g})$, $H^{CE}_{\ldot}(\mathfrak{g})$ are graded, unital, connected operads.
\end{remark}
Unfortunately, as we will see below, if the operad $\mathfrak{g}$ in $\Lie$ is finitely generated the space of $n$-ary operations $\mathfrak{g}(n)$ might be infinite-dimensional.
Consequently, the corresponding Chevalley-Eilenberg complex will be infinite-dimensional as well, however, the (co)homology might be finite-dimensional and assemble a finitely generated Hopf operad.
If in addition the algebras $\mathfrak{g}(n)$ are (pronilpotent) then we have the following splitting:
\[
\Com = H^{CE}_{0}(\mathfrak{g}) \rightarrow H^{CE}_{\ldot}(\mathfrak{g}) \rightarrow \Com
\]

\subsection{Simplest single-generated operad in Lie algebras}
\label{sec::Lie}
Consider an operad $\tkd^{k}$ in $\Lie$ generated by a single element $t_{[1k+1]}\in \tkd^{k}(k+1)$ (with symmetry relation $\sigma\cdot t_{[1k+1]} =(-1)^{k |\sigma|} t_{[1k+1]}$) and yielding the Leibniz rule:
\begin{equation}
\label{eq::Leibniz::Lie}
t_{[1k+1]}\circ_k (0)_{k+1,k+2} = (0)_{1,k+2} \circ_1 t_{[1k+1]} + (0)_{1,k+1}\circ_1 t_{[1k+2]\setminus \{k+1\}} 
\end{equation}
and this is the only condition we impose.

\begin{proposition}
	The space of $n$-ary operations in the operad $\tkd^{k}$ is the quadratic Lie algebra:
	\[
	\Lie\left(
	\begin{array}{c}
	\nu_{I}, \ I\subset [1n], |I|=k+1;\\
	\forall\sigma\in \bbS_{k+1}\ \nu_{\sigma(1)\tiny{\ldots}\sigma(k+1)} = (-1)^{|\sigma| k}\nu_{1\tiny{\ldots}k+1}\\
	\forall\tau\in A_{n}\subset \bbS_n \quad
	\tau(\nu_{I}) = \nu_{\tau(I)} 
	\end{array}
	\left|
	\begin{array}{c}
	[\nu_{I},\nu_{J}], \text{ if } I\cap J =\emptyset, \\ 
	{\left[\nu_{I},\sum_{s\in I}\nu_{J\setminus \{t\}\cup \{s\}}\right], 
		\text{ if } I\cap J =\{t\} }
	\end{array}
	\right.
	\right)
	\]
\end{proposition}
\begin{proof}
	The generator $\nu_I$ is the image of the generator $t_{[1 k+1]}$ with respect to appropriate operadic composition with all other inputs being zero:
	\[\nu_I  = 
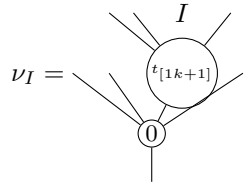
\begin{tikzpicture}[scale=0.8]
	\node[ext] (v) at (0,-1) {\small{$0$}};
	\node[ext] (w) at (0.5,0)  {\tiny{$t_{[1 k+1]}$}};
	\coordinate (s) at (0,-1.8);
	\coordinate (v0) at (-1.3,0);
	\coordinate (v1) at (-0.7,0);
	\coordinate (v2) at (1.5,0);
	\coordinate (u0) at (-0.3,1);
	\coordinate (u1) at (-0.7,1);
	\coordinate (u2) at (1.3,1);
	\node (s2) at (0.5,1) {{$I$}};
	\draw (v) edge (s) edge (v0)   edge (v1) edge (v2) edge (w);
	\draw (w) edge (u0)   edge (u1) edge (u2);
	\end{tikzpicture}
	\]
	The first quadratic relations follows from the commutativity of the monoidal structure in the operad of Lie algebras and the second relation is the application of the Leibniz rule~\eqref{eq::Leibniz::Lie}.
	
	It is easy to check that the aforementioned quadratic Lie algebras assemble an operad and consequently this operad coincides with $\tkd^{k}$ which is defined by universal properties of single generation that is obviously satisfied.
\end{proof}

\section{Single-generated Hopf operads}

\subsection{Operads of multiary commutative product and Lie bracket}
The operad $\Com$ is the first example of the quadratic Koszul operad generated by a single binary symmetric operation $\mu(x_1,x_2)=\mu(x_2,x_1)$.
The space of $n$-ary operations $\Com(n)$ is a one-dimensional trivial $\bbS_n$-representation with a chosen basic element 
\[\mu^{n-1}=\underbrace{\mu\circ\mu\circ\ldots\circ \mu}_{n-1} =\mu(x_{\sigma(1)},\mu(x_{\sigma(2)},\ldots \mu(x_{\sigma(n-1)},x_{\sigma(n)}), \ \forall \sigma\in \bbS_n
\] 

The homologically shifted operad $\Com\{d\}$ is a Koszul operad generated by a single binary operation $\mu(x_1,x_2) = (-1)^{d}\mu(x_2,x_1)$ yielding the (super)commutativity relation.
\begin{remark}
\label{rk::Com+}
	The operad generated by a single skew-symmetric binary operation $\mu(x_1,x_2)=-\mu(x_2,x_1)$ and yielding the associativity relations $\mu(\mu(x_1,x_2),x_3) = \mu(\mu(x_2,x_3),x_1) = \mu(\mu(x_3,x_1),x_2)$ is called $\Com_{-}$ and is not Koszul and even has no nontrivial operations in arity $4$ because:
\begin{multline*}
\mu(\mu(\mu(x_1,x_2),x_3),x_4) = - \mu(\mu(x_1,\mu(x_2,x_3)),x_4) =
\mu(x_1,\mu(\mu(x_2,x_3),x_4)) = \\
= - \mu(x_1,\mu(x_2,\mu(x_3,x_4))) = 
\mu(\mu(x_1,x_2),\mu(x_3,x_4)) = - \mu(\mu(\mu(x_1,x_2),x_3),x_4)
\end{multline*}
\end{remark}

Consider the natural analogues of the commutative operad denoted by $\Com^k$.
\begin{definition}
The operad $\Com^k$ is generated by a single $\bbS_{k+1}$-symmetric operation $\mu_{k+1}$ of arity $k+1$ and of $0$-homological degree subject to the following quadratic relations:
	\begin{multline}
\forall\sigma\in \bbS_{2k+1} \ \mu_{k+1}(\mu_{k+1}(x_{\sigma(1)},\ldots,x_{\sigma(k+1)}),x_{\sigma(k+2)},\ldots,x_{\sigma(2k+1)}) = \\
= \mu_{k+1}(\mu_{k+1}(x_1,\ldots,x_{k+1}),x_{k+2},\ldots,x_{2k+1}) 
\end{multline}
\end{definition}
\begin{proposition}
\label{lm::Com::k}	
The quadratic operad $\Com^k$ is Koszul and 
one has an isomorphism of $\bbS_{n}$-modules
 $$\Com^{k}(n) = \begin{cases}
\text{the trivial }\bbS_{n} \text{ representation } \mathbb{1}, \text{ if } (n-1) \vdots k, \\
0, \text{ if } (n-1) \not\vdots k
\end{cases}.$$ 
The corresponding generating series of dimensions has the following presentation
$$
\rchi_{\Com^k}(t) := \sum_{n\geq 1} \frac{\dim \Com^k(n)}{n!} t^n = \sum_{n\geq 1} \frac{t^{kn+1}}{(kn+1)!} =\frac{1}{k} \sum_{j=0}^{k-1} \xi^{-j} e^{\xi^j t}  
$$
where $\xi$ is the $k$'th primitive root of unity, e.g. $\xi = \exp(\frac{2\pi \sqrt{-1}}{k})$.
The symmetric function given by generating series of 
the corresponding $\bbS_{n}$-characters admits the following description in the basis of Newton power sums $p_m:=\sum x_i^m$:
\begin{equation}
\label{eq::Com::series}
\rchi_{\Com^k}(x_1,x_2,\ldots) = \sum_{n\geq 1} h_{nk+1} = \frac{1}{k}\left[ \sum_{j=0}^{k-1} \xi^{-j} \exp\left(\sum_{n\geq 1} \xi^{jn}\frac{p_n}{n}\right)\right]
\end{equation}
\end{proposition}
\begin{proof}
	The operad $\Com^{k}$ admits a quadratic Gr\"obner basis with respect to the path-lexicographical ordering such that the only normal quadratic monomial is $\mu(x_1,\ldots,x_k,\mu(x_{k+1},\ldots,x_{2k+1}))$, all other quadratic monomials are the leading monomials of the given Gr\"obner basis.  Consequently, $\Com^k$ is a Koszul operad thanks to the results of~\cite{DK::Grob}. See also \S5.2 and Theorem~5.3 of~\cite{DK::Anick}.
	
The Hilbert series are computed in a straightforward way, and follows from the clear description of the space of $n$-ary operations.
\end{proof}

 The homologically shifted operads $\Com^{k}\{d\}$ are Koszul as well. Note that $\Com^{k}\{d\}$  is generated by a single $k+1$-ary operation $\mu_{k+1}$ of homological degree $dk$ and for $\sigma\in\bbS_{k+1}$ we have $\sigma\cdot\mu_{k+1} = (-1)^{|\sigma| d}\mu_{k+1}$.
The Koszul-dual operad $\Lie^{k}\{-d\}:=(\Com^k\{d\})^{!}$ is generated by a single element $\nu_{k+1}= (-1)^{|\sigma| d} \sigma \cdot \nu_{k+1}$ of homological degree $1-dk$ yielding the following quadratic relation:
\begin{equation}
\label{eq::Jacobi}
\sum_{\sigma\in A_{2k+1}\cap \bbS_{2k+1}/\bbS_{k+1}\times\bbS_{k}}
\nu(\nu(x_{\sigma(1)},\ldots,x_{\sigma(k+1)}),x_{\sigma(k+2)},\ldots x_{\sigma(2k+1)}) = 0 
\end{equation}
The corresponding generating series of $\Lie^k$ is the inverse of the generating series of $\Com^k$ and its coefficients can be computed using, for example the Lagrange inverse formula:
\begin{equation}
\label{eq::Lie::series}
\dim\Lie^k(n) = 
\left[\left(\sum_{m\geq 0} \frac{(-t)^{km}}{(km+1)!} \right)^{-n}\right]_{(n-1)}.
\end{equation}
Here $[f(t)]_{(n)}$ denotes the $n$-th coefficient $f_n$ in the Taylor expansion of $f(t) = \sum_{n\geq 0} \frac{f_n}{n!}t^n$.
Moreover, one can use the presentation~\eqref{eq::Com::series} in terms of Newtons sums to get the inverse with respect to plethystic substitution that remembers the $\bbS$-character of $\Lie^k$.
Recall that for $k=1$ $\Lie^k=\Lie$ and the generating series $\rchi_{\Lie}(t) = -\ln(1-x)$, respectively $\rchi_{\Lie^2}(t)=\mathsf{arcsinh}(t)$.
However, we were not able to recognise the generating series $\rchi_{\Lie^k}(t)$ as a known function for $k>2$.
 
\begin{remark}
	The operad $\Com^{k}_{-}$ (as well as its homological shifts $\Com^{k}_{-}\{d\}$) that are generated by a single skew-symmetric operation are obviously not Koszul and thus they are out of our consideration. We refer to~\cite{Markl_Rem_nonkoszul} for the corresponding nonsymmetric analogues of these operads.
\end{remark}

\subsection{Generalizations of the Poisson operad}
\label{sec::Pois::def}
Let us define a higher-dimensional analogue of the operad of Poisson algebras. Let $\Pois^k$ be an algebraic operad generated by a commutative associative binary multiplication $\mu_2$ and a skew-symmetric $k+1$-ary operation $\nu_{k+1}$ of homological degree $1-k$ that generates the suboperad $\Lie^{k}:= (\Com^k)^{!}$ subject to the following Leibniz identity:
\begin{multline}
\label{eq::Leib::rel}
\nu_{k+1}(x_1,\ldots,x_k,\mu_2(x_{k+1},x_{k+2})) =  \\
= \mu_2(\nu_{k+1}(x_1,\ldots,x_k,x_{k+2}),x_{k+1})+ (-1)^{\epsilon} 
\mu_2(\nu_{k+1}(x_1,\ldots,x_k,x_{k+1}),x_{k+2}).
\end{multline}
Here $\epsilon=|x_{k+1}| \cdot (1-k+|x_1|+\ldots+|x_k|)$.
In other words, the Leibniz identity says that the higher Lie bracket with all arguments fixed (except one) defines a derivation of the commutative associative product:
\[
\text{Let }D(\ttt):=\nu(x_1,\ldots,x_k,\ttt) \text{ then } D(\mu(a, b)) = \mu(D(a), b) + (-1)^{|D| |a|}\mu(a, D(b)).
\]
Let $\Pois^{k}_d$ be the operad that generalizes the homology of the little discs operad. That is, $\Pois^{k}_d$ is generated by degree $0$ binary commutative multiplication $\mu$ and $k+1$-ary (skew)-symmetric operation $\nu$ that generates the shifted operad $\Lie^k\{1-d\}$ yielding the Leibniz identity~\eqref{eq::Leib::rel}.
\begin{lemma}
The following map of generators:
\begin{equation}
\label{eq::Delta::Pois:k}
\Delta(\mu_2) = \mu_2\otimes \mu_2, \quad 
\Delta(\nu_k) = \mu_2^{\circ k} \otimes \nu_k + 
\nu_k\otimes \mu_2^{\circ k}
\end{equation}
uniquelly extents to a Hopf structure $\Delta:\Pois^k_d\rightarrow \Pois^k_d\otimes \Pois^k_d$ on the operad $\Pois^k_d$.	
\end{lemma}
\begin{proof}
In order to show that $\Delta:\Pois^k\rightarrow \Pois^k\otimes \Pois^k$ defines a morphism of operads one has to check that the comultiplication of the relations of $\Pois^k$ is zero. Let us show that this indeed happens for the generalized Jacobi identity (for all other relations it is obvious):
\begin{multline}
\label{eq::Delta::Jacobi}
\Delta\left(\sum_{\begin{smallmatrix}
	I\sqcup J = [1 2k+1], \\
	|I|+1=|J|=k+1
	\end{smallmatrix}}\pm \nu(x_{I},\nu(x_J))\right) = 
\sum_{\begin{smallmatrix}
	I\sqcup J = [1 2k+1], \\
	|I|+1=|J|=k+1
	\end{smallmatrix}}
\left[\pm \nu(x_{I},\nu(x_J))\otimes \mu^{\circ 2k+1} + \phantom{\frac{1}{2}}\right.
\\
\pm \left( (\nu(x_{I},\ttt)\otimes \mu^{\circ 2k+1} )
\circ (\mu^{\circ 2k+1} \otimes \nu({x_{J}}))
\right)
 + \left((\mu^{\circ 2k+1} \otimes \nu({x_{I},\ttt})) \circ (\nu(x_{J})\otimes \mu^{\circ 2k+1} ) \right) + 
\\
\left.\phantom{\frac{1}{2}} \pm \mu^{\circ 2k+1} \otimes \nu(x_{I},\nu(x_J)) \right].
\end{multline}
The first and the last summands in~\eqref{eq::Delta::Jacobi} dissapear because of the Jacobi identity in each tensor multiple of $\Pois^k\otimes\Pois^k$. Let us expand middle terms  using the Leibniz rule. 
Each summand in the expansion will be a tensor product of shuffle monomials of the form:
\begin{equation}
\label{eq::sum::leib}
\left(\mu^{\circ 2k+1}\circ \nu(x_I)\right)\otimes\left( \mu^{\circ 2k+1}\circ \nu(x_J)\right) \text{ with } |I\cap J| = 1, \ |I|=|J|=k+1.
\end{equation}
For all pairs of subsets $I$ and $J$ the term~\eqref{eq::sum::leib} appears in the expansion of~\eqref{eq::Delta::Jacobi} twice and thanks to the symmetry of the group $\bbS_{2k+1}$ we can check that corresponding coefficients has opposite signs for a chosen $I=\{1,\ldots,k+1\}$ and $J=\{k+1,\ldots,2k+1\}$.
Indeed, one monomial comes out from the expansion of the second term in the summand of~\eqref{eq::Delta::Jacobi} for $I=\{1,\ldots,k\}$, $J=\{k+1,\ldots,2k+1\}$ and the of expansion of the 
third term of the summand of~\eqref{eq::Delta::Jacobi} for $I=\{k+1,\ldots,2k\}$ and $J=\{2k+1,1,\ldots,k\}$ gives the second nontrivial contribution. The sign is affected by the Koszul sign rule while interchanging the tensor product (the degree of $\nu$) and by the sign of the long cycle $(1 2 \ldots k 2k+1)$ and we get 
$(-1)^{1-k}(-1)^{k} = -1$. 

The coassociativity of the comultiplication $\Delta$ is also enough to check on the level of generators where it is obvious. Note that $\mu$ as well as its iterated compositions $\mu^{\circ N}$ plays a role of a counit in $\Pois^k(n)$.
\end{proof}

\begin{remark}
	The operad $\Pois^k_d$ is generated by a single primitive element $\nu$ and admits a $0$ ary operation $\kappa$ such that $\mu\circ_1\kappa= Id$ and $\nu\circ_1\kappa=0$.
\end{remark}

\begin{proposition}
\label{thm::Pois::Koszul}	
\begin{itemize}
\item The Leibniz identity~\eqref{eq::Leib::rel} defines a distributive law in the sence of~\cite{Markl_distr_law}:
\[\Lie^{k}\{1-d\} \circ \Com \rightarrow \Com \circ \Lie^{k}\{1-d\} 
\]
between Koszul operads $\Com$ and $\Lie^{k}\{1-d\}$.  
\item The quadratic operad $\Pois^k_d$ is Koszul.
\item There exists an isomorphism of symmetric collections $\Pois^k_d$ and $\Com \circ \Lie^{k}\{1-d\}$.
In particular, we have the following description of generating series:
\[\rchi_{\Pois^k}(t,q):= \sum_{n\geq 1} \frac{\dim_q(\Pois^k(n))}{n!} t^n = \exp{\left(\frac{\rchi_{\Lie^k}(qt)}{q}\right)}-1\]
\end{itemize}	
\end{proposition}
\begin{proof}
	The proof is standard and repeats the one known for Poisson operad $\Pois$ (see e.g.~\cite{Markl_distr_law}). Moreover, one can define a quadratic Gr\"obner basis in $\Pois^k$ using the methods of~\cite{Dotsenko_Distributive_Law}.
\end{proof}

\subsection{Resolving the Lie bracket in $\Pois^k_d$}
\label{sec::hoPois}
Consider the following Hopf dg-operad $\ho\Pois^k_d$ generated by the commutative associative product $\mu$ and a collection of operations $\nu_{nk+1}$, $n=1,2,\ldots$ each yielding the Leibniz rule~\eqref{eq::Leib::rel} with the multiplication $\mu$. The differential acts nontrivially only on operations $\nu_{\ldot}$:
\[d(\nu_{nk+1}) := \sum_{i+j=n}\pm\nu_{ik+1}\circ\nu_{jk+1}, \quad d(\mu)=0.\]
The dg-suboperad generated by all $\nu_{\ldot}$ is isomorphic to the Koszul resolution $\Omega(\pdu\Com^k_d)\twoheadrightarrow\Lie^{k}\{1-d\}$ and is denoted by $\LL_\infty^k\{1-d\}$.

\begin{theorem}
	The surjection $\ho\Pois^k_d\twoheadrightarrow \Pois^k_d$ is a quasiisomorphism. 
\end{theorem}
\begin{proof}
First, It is easy to show that the Leibniz rule defines a distributive law  between the operad $\Com$ and the (dg) operad $\LL_\infty^{k}\{1-d\}$.
Generalizing the known basis and ordering of shuffle monomials for $\Pois$ suggested in~\cite{Dotsenko_Distributive_Law} one can easily find a quadratic Gr\"obner basis for the operad  $\ho\Pois^k_d$.
Second, we already know from the previous section that Leibniz rule defines a distributive law on $\Pois^k_d$.
Third, thanks to Koszulness of $\Lie^k$ we know that there is a quasi-isomorphism $\LL_\infty^{k}\twoheadrightarrow \Lie^k$.
what implies the quasi-isomorphism:
\begin{equation}
\label{eq::com::lie::quis}
\Com\circ(\LL_{\infty}^k\{1-d\}) \twoheadrightarrow \Com\circ\Lie^k_d
\end{equation}
Note that the source and the target of the quasiiso~\eqref{eq::com::lie::quis} are the associated graded with respect to the appropriate filtrations on the operads $\ho\Pois^k_d$ and $\Pois^k_d$ that are predicted by distributive law. Consequently, they are quasiisomorphic.
\end{proof}

\subsection{Cofibrant dgca model of $\Pois^k_d$ via oriented graphs and koszulness of $\tkd^k_d(n)$}
\label{sec::Graphs::odd}
As mentioned in the title we suggest a combinatorial models of $\Pois^k_d$ that are cofibrant meaning that each space of $n$-ary operations is a free dg cocommutative coalgebra. These models are used to show the koszulness of $\Pois^k_d(n)$ and compare them with the Chevalley-Eilenberg complex of $\tkd_d^k(n)$.

Let us start from combinatorial definitions:
\begin{definition}
We say that a graph $\Gamma$ is \underline{$k$-mod oriented} iff $\Gamma$   is a graph whose edges are oriented, such that in addition
\begin{itemize}
\setlength{\itemsep}{-0.3em}	
	\item $\Gamma$ has $n\geq 1$ external nubered (white) vertices drawn as rounded numbers $\extv{s}$;
	\item there are no edges starting in an external vertex;
	\item with arbitrary amount of unordered (black) vertices;
	\item for all internal vertex $v$ the remainder modulo $k$ of the number of outgoing edges from $v$ is equal to $1$;
	\item there are no directed cycles;
	\item we do not allow vertices with one output and at most one input.
\end{itemize}
We say that the homological degree of $\begin{array}{l}
\bullet \text{an edge is equal to $1-d$} \\
\bullet \text{an internal vertex is equal to $d$;}\\
\bullet \text{an external vertex is equal to $0$.}
\end{array}$ 
\\
The homological degree affects the symmetry group and the full homological degree of a graph.
\end{definition}
The linear span of $k$-mod oriented graphs is denoted by $\fGraphs^{k\orient}_{d}(n)$.\footnote{We use index $d$ because the graphs corresponds to the differential forms in the upper-half space of $\mathbb{R}^{d}$ what helps to remember the gradings, see e.g.~\cite{Willwacher_oriented,Khor::Wilw::MonR}.} 
The linear dual space (with homological degree reversed) has the same basis given by $k$-mod oriented graphs and will be denoted by $\pdu\fGraphs^{k\orient}_{d}(n)$. 
Suppose $\Gamma_1$ is a $k$-mod oriented subgraph on $m$ vertices in a $k$-mod oriented graph $\Gamma$ on $n$ vertices such that if the vertex $v$ belongs to $\Gamma_1$ then all its outgoing edges also belongs to $\Gamma_1$, then if we contract a subgraph $\Gamma_1$ and replace it by an external vertex the resulting graph $\Gamma/\Gamma_1$ is also a $k$-mod oriented graph with $n-m+1$ external vertices.
This operation defines a cooperad structure on $\pdu\fGraphs^k_{d}$.
The dual operation defines an operad structure on $\fGraphs^k_{d}$.

The edge contraction operation defines a differential on the cooperad 
$\pdu\fGraphs^k_{d}$. 
The operad $\fGraphs^k_{d}$ is a dg-operad with the linear dual differential given  pictorially by vertex splitting differential as it is in all graph complexes and operads of Graphs used in~\cite{Kont_Motives,Lamb_Vol,Willwacher_grt,Willwacher_oriented}.
Moreover, every pair of graphs in $\pdu\fGraphs^k_d$ can be glued through the external vertices and we end up with the Hopf (co)operad structure on $(\pdu)\fGraphs^k_d$. 
As suggested in~\cite{Severa_Willwacher} let us denote by $\ICG^{k\orient}$ the subspace of internally connected $k$-mod oriented graphs.  The collection $\ICG^{k\orient}(n)$ is an operad in the category of $\LL_{\infty}$-algebras, where the $\LL_{\infty}$ structure is prescribed by the isomorphism of  $\fGraphs^k_d(n)$ and the  Chevalley-Eilenberg complex $C_{CE}(\ICG^{k\orient}(n))$.

\begin{theorem}
\label{thm::H:Graph:1}	
	The following assignment of a graph to each generator of $\ho \Pois^k_{d}$
	\begin{equation}
	\label{eq::map::e_d::Graphs::odd}
	\ho\Pois^k_{d}(2)\ni \mu_2 \mapsto  
	\begin{tikzpicture}[scale=.5]
	\node[ext] at (0,0) {};
	\node[ext] at (1,0) {};
	\end{tikzpicture},
	\quad
	\ho\Pois_{d}(mk+1)\ni \nu_{mk+1} \mapsto  
	\begin{tikzpicture}[scale=1]
	\node[int] (v) at (0,1) {};
	\node[ext] (v0) at (-1.3,0) {1};
	\node[ext] (v4) at (1.3,0) {\tiny{mk+1}};
	\node[ext] (v1) at (-.7,0) {2};
	\node (v2) at (0,0) {$\cdots$};
	\node[ext] (v3) at (.7,0) {\tiny{mk}};
	\draw[-triangle 60] (v) edge (v1)  edge (v3) edge (v0) edge (v4);
	\end{tikzpicture}
	\end{equation}
	extends to a quasiisomorphism of Hopf operads $\ho \Pois^k_{d} \to \fGraphs_{d+1}^{k\orient}$.
\end{theorem}	
\begin{proof}
	The proof of Theorem~\ref{thm::H:Graph:1} generalizes the one given in~\cite{Willwacher_oriented} and is postponed to the Appendix~\S\ref{sec::proof::Graphs}.
\end{proof}

Let us suppose that $k\geq 2$ since the $k=1$ leads to the classical story of the little discs operad (\cite{Willwacher_oriented}). The case $k=2$ was covered in~\cite{Khor::Wilw::MonR} and we are going to suggest the direct generalization of all arguments suggested in that paper.

First of all, let us recall that all inner vertices of a graph $\Gamma\in\ICG^{k\orient}$ as well as $\Gamma\in\fGraphs^{k\orient}$ may not have univalent and bivalent vertices. Moreover, if an internal vertex has more than $1$ outgoing edges then it has at least $k+1$ outgoing edges. If the vertex has exactly one outgoing edge than it has at least two incoming edges. Moreover, in addition to the homological grading there exists a grading on $\ICG^{k\orient}$ given by the loop order divided by $k$:
$$\deg(\Gamma): = \frac{\# \text{edges} - \#\text{internal vertices}}{k}$$ 
that respects the $\LL_\infty$-structure.
The same grading exists for internally disconnected graphs $\fGraphs^{k\orient}_d$ as well and Theorem~\ref{thm::H:Graph:1} predicts that the graded component $(\ICG^{k\orient}_{d+1})_{m}$ of degree $m$ belongs to the homological degree $m(1-kd)$.
Notice that the homological and internal gradings  of the Lie algebra $\tkd^k_d(n)$ differ by analogous linear transformation and we are ready to state one of the main theorems of this note:

\begin{theorem}
\label{thm::ICG::Koszul}	
\begin{enumerate}
	\item 
	The assignment 
\begin{equation}
\label{eq::HICG::generators}
\psi: t_{[1k+1]} \mapsto
	\begin{tikzpicture}[scale=1]
	\node[int] (v) at (0,1) {};
	\node[ext] (v0) at (-1.3,0) {1};
	\node[ext] (v4) at (1.3,0) {\tiny{k+1}};
	\node[ext] (v1) at (-.7,0) {2};
	\node (v2) at (0,0) {$\cdots$};
	\node[ext] (v3) at (.7,0) {\tiny{k}};
	\draw[-triangle 60] (v) edge (v1)  edge (v3) edge (v0) edge (v4);
	\end{tikzpicture}
\end{equation}	
extends uniquely to the isomorphism of the operads $\psi: \tkd^k_d \rightarrow H^{\udot}(\ICG^{k\orient}_{d+1})$ in the category of graded $\LL_\infty$-algebras.  
\item The Quadratic Lie algebras $\tkd^k_d(n)$ are Koszul. 
\end{enumerate}
\end{theorem}
The core of the proof is hidden in the following technical Lemma~\ref{lem::ICG::trivalent} whose proof is postponed to the Appendix~\S\ref{sec::proof::ICG}.	
\begin{lemma}
	\label{lem::ICG::trivalent}
	All non-trivial cohomology classes in $\ICG^{k\orient}$ can be represented by linear combinations of internally trees whose all inner vertices are either has one output and two inputs or has $k+1$ outgoing edges and no incoming edges. 
	In particular, for each loop degree $m$ there exists a unique nonvanishing homological component of $H(\ICG^{k\orient})_{m}$.
\end{lemma}
Let us explain how Lemma~\ref{lem::ICG::trivalent} implies Theorem~\ref{thm::ICG::Koszul}.
\begin{proof}
First, let us notice that the gradings bounds implies vanishing of all higher Lie brackets on the homology  $H^{\udot}(\ICG^{k\orient}(n))$. Second, we see from  the recursive description of cycles representing cohomology $H^{\udot}(\ICG^{k\orient}(n))$ that the corresponding Lie algebra is one generated and the generators are given by the image of generators of $\tkd(n)$ described in the assignment~\ref{eq::HICG::generators}. 
Third, the relations in the Lie algebra $\tkd^k(n)$ are easily verified in $H^{\udot}(ICG^{k\orient}(n))$ as a boundary of a simplest graph with one inner edge. Let us give the pictorial presentation for $k=2$:
\begin{align}
\label{eq::HICG::relations}
d\left(
\begin{tikzpicture}[scale=0.5]
\node[int] (v) at (-1,1.2) {};
\node[int] (w) at (1,1.2) {};
\node[ext] (v1) at (-2,0) {i};
\node[ext] (v2) at (-1,0) {\small{j}};
\node[ext] (v3) at (0,0) {k};
\node[ext] (v4) at (1,0) {\small{p}};
\node[ext] (v5) at (2,0) {\small{q}};
\draw[-triangle 60] (v) edge (v1)  edge (v2) edge (v3);
\draw[-triangle 60] (w) edge (v)  edge (v4) edge (v5);
\end{tikzpicture}
\right)
& =  
\begin{tikzpicture}[scale=0.5]
\node[int] (v) at (-1,1.2) {};
\node[int] (w) at (1,1.2) {};
\node[int] (w1) at (0,1) {};
\node[ext] (v1) at (-2,0) {i};
\node[ext] (v2) at (-1,0) {\small{j}};
\node[ext] (v3) at (0,0) {k};
\node[ext] (v4) at (1,0) {\small{p}};
\node[ext] (v5) at (2,0) {\small{q}};
\draw[-triangle 60] (v) edge (v1)  edge (v2) edge (w1);
\draw[-triangle 60] (w) edge (w1)  edge (v4) edge (v5);
\draw[-triangle 60] (w1) edge (v3);
\end{tikzpicture}
\ + \
\begin{tikzpicture}[scale=0.5]
\node[int] (v) at (-1,1.2) {};
\node[int] (w) at (1,1.2) {};
\node[int] (w1) at (0,1) {};
\node[ext] (v1) at (-2,0) {\small{j}};
\node[ext] (v2) at (-1,0) {k};
\node[ext] (v3) at (0,0) {i};
\node[ext] (v4) at (1,0) {\small{p}};
\node[ext] (v5) at (2,0) {\small{q}};
\draw[-triangle 60] (v) edge (v1)  edge (v2) edge (w1);
\draw[-triangle 60] (w) edge (w1)  edge (v4) edge (v5);
\draw[-triangle 60] (w1) edge (v3);
\end{tikzpicture}
\ + \
\begin{tikzpicture}[scale=0.5]
\node[int] (v) at (-1,1.2) {};
\node[int] (w) at (1,1.2) {};
\node[int] (w1) at (0,1) {};
\node[ext] (v1) at (-2,0) {k};
\node[ext] (v2) at (-1,0) {i};
\node[ext] (v3) at (0,0) {\small{j}};
\node[ext] (v4) at (1,0) {\small{p}};
\node[ext] (v5) at (2,0) {\small{q}};
\draw[-triangle 60] (v) edge (v1)  edge (v2) edge (w1);
\draw[-triangle 60] (w) edge (w1)  edge (v4) edge (v5);
\draw[-triangle 60] (w1) edge (v3);
\end{tikzpicture}
= \\
& = 
[\psi(\nu_{ijk}),\psi(\nu_{kpq})] 
\ + \
[\psi(\nu_{ijk}),\psi(\nu_{ipq})]
\ + \
[\psi(\nu_{ijk}),\psi(\nu_{jpq})]. 
\end{align}
Consequently, the map $\psi:\tkd^k_d(n)\to H^{\udot}(\ICG^{k\orient}(n))$ is a surjective map of Lie algebras. 
Fourth, thanks to Theorem~\ref{thm::H:Graph:1} we already know that $H^{CE}(\ICG^{k\orient}(n)) = H(\fGraphs^{k\orient}(n)) = \Pois_d^k(n)$. In particular, the $m$-th cohomology differs from zero only in the $m$-th inner (=loop) degree. Therefore, the Lie algebra $\ICG^{k\orient}(n)$ is Koszul and, hence, quadratic. It remains to show that the set of quadratic relations coincide with the image of the known relations for $\tkd^k(n)$.
It is easy to find a bijection between quadratic shuffle monomials in $\Lie^{k}(2k+1)$ and the set of quadratic relations in $\tkd^k_d(2k+1)$ which is drawn by the picture on the left hand side of~\eqref{eq::HICG::relations}. The generalized Jacobi identity in $\Lie^k(2k+1)$  corresponds to the unique linear dependence for the aforementioned quadratic relations in $H^{\udot}(\ICG^{k\orient}(2k+1))$ :
\begin{equation}
d\left(
\begin{tikzpicture}[scale=0.8]
	\node[int] (v) at (0,1) {};
	\node[ext] (v0) at (-1.3,0) {1};
	\node[ext] (v4) at (1.3,0) {\tiny{2k+1}};
	\node[ext] (v1) at (-.7,0) {2};
	\node (v2) at (0,0) {$\cdots$};
	\draw[-triangle 60] (v) edge (v1)  edge (v0) edge (v4);
\end{tikzpicture}
\right) =
\sum_{\sigma\in S_{2k+1}} \pm 
\begin{tikzpicture}[scale=0.8]
	\node[int] (v) at (0,1) {};
	\node[ext] (v0) at (-1.3,0) {\tiny{$\sigma_1$}};
	\node[ext] (v4) at (1.3,0) {\tiny{$\sigma_{k+1}$}};
	\node[ext] (v1) at (-.7,0) {\tiny{$\sigma_{2}$}};
	\node (v2) at (0,0) {$\cdots$};
	\draw[-triangle 60] (v) edge (v1)  edge (v3) edge (v0) edge (v4);
		\node[int] (w) at (4+0,1) {};
		\node[ext] (w0) at (4-1.3,0) {\tiny{$\sigma_{k+2}$}};
		\node[ext] (w4) at (4+1.3,0) {\tiny{$\sigma_{2k+1}$}};
		\node (w2) at (4+0,0) {$\cdots$};
		\draw[-triangle 60] (w) edge (w0) edge (w2) edge (w4) edge (v);
\end{tikzpicture}
\end{equation}
What follows, that  the quadratic relations in the Lie algebra $\tkd^k_d$ spans the whole set of quadratic relations for those image. 
Therefore, $H(\ICG^{k\orient})(n)$ and $\tkd^k_d(n)$ are isomorphic as graded Lie algebras and, in particular, $\tkd^k_d(n)$ is a quadratic Koszul Lie algebra.
\end{proof}

\subsection{Generalizing Orlik-Solomon algebras}
Let us define a collection of quadratic algebras that generalizes the Orlik-Solomon algebras (\cite{Orlik}).
For a given positive integer $k$ and a finite set $I$ let $\calA^k_{d}(I)$ be a graded algebra generated by elements $\omega_{S}$ indexed by linearly ordered subsets $S\subset I$ of cardinality $k+1$ subject to the following quadratic relations:
\begin{gather}
\label{eq::OS::multi}
\omega_{S} \omega_{T} = 0 \text{ if } |S\cap T| > 1, \\
\label{eq::OS::cyclic}
\omega_{\{i_1 \ldots i_{k+1}\}} \omega_{\{i_{k+1} \ldots i_{2k+1}\}} +
\text{ $\ZZ_{2k+1}$-cyclic permutations } = 0
\end{gather}
We pose the homological degree of $\omega_S$ to be equal $kd-1$ and $\omega_{\sigma(S)}=(-1)^{d|\sigma|}\omega_{S}$ for $\sigma\in \bbS_{k+1}$.

\begin{remark}
	The difference of two relations~\eqref{eq::OS::cyclic} with two consecutive indices interchanged leades to  the following identity for any decomposition  $I\sqcup J\sqcup\{a,b,c\}$ of a set of cardinality  $2k+1$: 
	\begin{equation}
	\label{eq::OS::small}
	\omega_{a,I,c}\omega_{c,J,b}-(-1)^d\omega_{b,I,c}\omega_{c,J,a} + \omega_{c,J,b}\omega_{b,I,a} - (-1)^d \omega_{c,J,a}\omega_{a,I,b} = 0
	\end{equation}
\end{remark} 

\begin{proposition}
	For a given $k,d$ the collection $\calA^{k}_d([1n])$ assembles a structure of the Hopf cooperad subject to the cocomposition  $\varphi_{\calA}^{I,J}: \calA(I\sqcup J)\rightarrow \calA(I\sqcup\{*\})\otimes \calA(J)$ defined on the generators by the following rule:
	\begin{equation}
	\label{eq::cocompos::OS}
	\varphi_{\calA}^{I,J}: 
	\omega_{S} \mapsto 
	\begin{cases}
	\omega_S\otimes 1 \text{ if } S\subset I, \\
	1\otimes \omega_S \text{ if } S\subset J, \\
	\omega_{S\setminus\{s\}\sqcup \{*\}} \text{ if } S\cap J = \{s\}, \\
	0, \text{ if } |S\cap I|\geq 1 \ \& \ |S\cap J|\geq 2.
	\end{cases}
	\end{equation}
\end{proposition}
\begin{proof}
	It is a straightforward check that cocomposition preserves the relations. Let us illustrate it for the most interesting case $I=\{1,\ldots,k\}$, $J=\{k+1,\ldots,2k\}$ and the relation~\eqref{eq::OS::cyclic}:
	\begin{multline*}
	\varphi_{\calA}^{I,J}\left(\sum_{m=0}^{2k} \omega_{[m+1]_{2k+1},\ldots,[m+k+1]_{2k+1}} \omega_{[m+k+1]_{2k+1},\ldots,[m+2k+1]_{2k+1}}\right) = \\
	= (\omega_{1,\ldots,k,*}\otimes 1)\cdot(1\otimes \omega_{k+1,\ldots,2k+1})  + (1\otimes \omega_{k+1,\ldots,2k+1}) \cdot (\omega_{*,1,\ldots,k} \otimes 1) =   \\
	= \omega_{1,\ldots,k,*}\otimes \omega_{k+1,\ldots,2k+1} + (-1)^{kd-1 +d(k)}  \omega_{1,\ldots,k,*}\otimes  \omega_{k+1,\ldots,2k+1} = 0
	\end{multline*}
	Here $[l]_{2k+1}$ denotes the remainder of $l$ modulo $2k+1$.
\end{proof} 

\begin{example}
	\begin{itemize}
		\item 
		The algebra	$\calA^{k}_d(I)$ is isomorphic to the ground field $\kk$ if the cardinality of $I$ is less then $k+1$.
		\item 
		The algebra $\calA^k_d(k+1)$ is two-dimensional and has a basis $\{1,\omega_{1,\ldots,k+1}\}$.
	\end{itemize}
\end{example}

\begin{lemma}
	The following set of quadratic monomials 
	\begin{equation}
	\label{eq::basis::A_2}
	\calB(2k+1):=
	\left\{ \omega_{\{1\}\sqcup I} \omega_{ J} \left| 
	\begin{array}{c}
	I\cap J = \{s\}, I\cup J\cup\{1\}= \{1,\ldots,2k+1\} \\
	s = \min(J) \text{ or } s = \max(I).
	\end{array}\right.
	\right\} 
	\end{equation}
	spans the second graded component of the algebra $\calA^k(2k+1)$.	
\end{lemma}
\begin{proof}
	Thanks to the relation~\eqref{eq::OS::multi} we know that the second graded component is spanned by monomials $\omega_{I} \omega_{J}$ with $|I\cap J| = 1$ and $|I\cup J| =2k+1$.
	Using the relation~\eqref{eq::OS::cyclic} we can even suppose that $1\notin I\cap J$.
	Let us call the intersection $I\cap J$ by the repeated element and it remains to use the relation~\eqref{eq::OS::small} to make this repeated element either the maxima of $I$ or the minima of $J$.
	We can do it inductively, because  for a monomial $\omega_{1\sqcup I} \omega_{J}\notin \calB(2k+1)$ there exists a triple $a\in I\cap J$, $b\in J$, $c\in I$ such that $c>b>a$:
	\[
	\omega_{1\sqcup I} \omega_{J} = \pm \omega_{1\sqcup I} \omega_{\{c\}\sqcup J\setminus\{a\} }\pm  \omega_{1\sqcup \{b\}\sqcup I\setminus \{c\}} \omega_{ J} 
	\pm \omega_{1\sqcup \{b\}\sqcup I\setminus \{c\}} \omega_{\{c\}\sqcup J\setminus\{b\}}.
	\]
	and we see that in the first and in the third summands of the right hand side we increased (in the local order of $I$) the number of the repeated element, and in the second summand we decreased (in the local order of $J$) the  number of the repeated element. 
\end{proof}

\begin{corollary}
	There is an upper bound on the dimensions of the component of degree $2$ in the algebra $\calA^k_d$:
	\[\dim\calA^k_d(2k+1)_{(2)}\leq \dim \Lie_d^k(2k+1)\]
\end{corollary}
\begin{proof}
	It is straightforward to verify that the map $\varphi$ below between shuffle monomials in $\Lie(2k+1)$ that are not divisible by $\nu(\nu(x_1,\ldots,x_{k+1}),x_{k+2},\ldots,x_{2k+1})$ and the elements of the set $\calB(2k+1)$ is bijective:
	\begin{align}
	\label{eq::Lie->A:1}
	\varphi: \nu(\nu(x_1,x_{i_1},\ldots,x_{i_k}), x_{j_1},\ldots,x_{j_k}) & \mapsto  \omega_{\{1,i_1,\ldots,i_k\}} \omega_{\{i_k,j_1,\ldots,j_k\}} \\
	\label{eq::Lie->A:2}
	\varphi: \nu(x_1,x_{i_1},\ldots,x_{i_{s-1}},\nu(x_{i_s},x_{j_1},\ldots,x_{j_{k}}), x_{i_{s+1}},\ldots,x_{i_k}) & \mapsto  \omega_{\{1,i_1,\ldots,i_k\}} \omega_{\{i_s,j_1,\ldots,j_k\}}. 
	\end{align}	
\end{proof}

\subsection{Nondegenerate pairing between operad $\Pois^{k}_{d}$ and cooperad $\calA^k_d$}
\label{sec::pairing}
The goal of this section is to recognize the Hopf structure on $\Pois^k_d$, namely, we want to prove the isomorphism $\pdu\Pois^k_d(n)\simeq \calA^k_d(n)$.

\begin{theorem}
\label{thm::pairing}	
The Hopf cooperad $\calA^k_d$ is dual to the Hopf operad $\Pois^k_d$. In particular,
	\begin{enumerate}	
		\item
		The pairing of the generators of the the Hopf cooperad $\calA$ and the Hopf operad $\Pois_d^k$:
		\[	\langle  1, \mu \rangle = \langle  \omega_{S} , \nu(x_S) \rangle = 1 , \  \langle \omega_{S} , \mu^{\circ k} \rangle = \langle 1 , \nu \rangle = 0
		\]
		extends uniquely 
		to a \underline{nondegenerate} pairing between the Hopf (co)operads that is compatible with the Hopf (co)operads structure in the sense of Lemma~\ref{lem::pairing}.
		\item	
		Moreover,  the pairing between an operadic monomial $\gamma\in \Pois^k_d(I)$ and a monomial  $f:=\omega_{S_1}\cdot \ldots \omega_{S_m}$ in the algebra $\calA^d_k(I)$ differs from zero if and only if there exists a bijection $\psi$ between vertices of $\gamma$ labelled by $\nu$ and the generators $\omega_{S_t}$ dividing $f$ such that for all $t=1\ldots m$ there exists an ordering of the elements of the subset $S_t:=\{s_1<\ldots<s_{k+1}\}$ such that 
		the leaf $s_i$ of $\gamma$ belongs to the subtree growing from the $i$-th input of the vertex $\psi(\omega_{S_t})$.  
	\end{enumerate}
\end{theorem}
\begin{proof}
	The second part of the theorem follows directly from the compatibility of the pairing and the Hopf operad structure. 
	On the other hand these vanishing properties of monomials explains the straightforward computation that shows that quadratic relations in the operad $\Pois^k_d$ and quadratic relations in algebra $\calA^k_d$ belongs to the kernel of this pairing what follows that the pairing is well defined.
	E.g. there are either $0$ or two monomials (with opposite signs) from the l.h.s. of~\eqref{eq::OS::cyclic} that have nonzero pairing with a shuffle monomial $\nu(x_{I},\nu(x_{J}))$. Similarly, there are either $0$ or two shuffle monomials (with opposite signs) in the Jacobi identity, that has nonzero pairing with a monomial $\omega_{S}\omega_{T}$.
	
Thanks to Theorem~\ref{thm::H:Graph:1} and Theorem~\ref{thm::ICG::Koszul}	 we already know that the coalgebra $\Pois^k_d(n)$ is a quadratic Koszul algebra whose Koszul dual is isomorphic to $\tkd^k_d(n)$. Thus, it remains to show that the aforementioned pairing is nondegenerate on the level of quadratic components where one can easely see that the matrix of pairing between monomial shuffle basis in $\Lie^k_d(2k+1)$ and the spanning set $\calB(2k+1)$ of $\calA^k_d(2k+1)$ is diagonal:
\[\langle \gamma, \varphi(\gamma')\rangle= \delta_{\gamma,\gamma'} \]
 The bijection $\varphi$ was defined in \ref{eq::Lie->A:1},\ref{eq::Lie->A:2}.   
\end{proof}

While forgetting the operad structure in the isomorphism $\pdu\Pois^k_d(n)\simeq \calA^k_d(n)$ we end up with one of the main applications:
\begin{corollary}
	\label{thm::dual::OS::KD}
	The quadratic (super)commutative generalized Orlik-Solomon algebra $\calA^{k}_d(n)$ and the quadratic generalized Drinfeld-Kohno Lie algebra $\tkd_d^k(n)$ are quadratic dual to each other and satisfy the Koszul property.
\end{corollary}

\subsection{A basis in generalized Orlik-Solomon algebras $\calA^k_d(n)$}
\label{sec::OS::basis}
Let us use the nondegeneracy of the pairing discovered in~\ref{thm::pairing} and present a basis in $\calA^k_d(n)$ that is dual to the operadic basis in $\Pois_d^k(n)$. 

\begin{lemma}
	\label{lem::shuffle::span}	
	The set $\shufLie(n)$ consisting of shuffle monomials generated by a single $k+1$-ary operation $\nu$ that are not divisible by the shuffle monomial $\nu(\nu(x_1,\ldots,x_{k+1}),x_{k+2},\ldots,x_{2k+1})$ form a basis of the space of $n$-ary operations of the operad $\Lie^k$. 
\end{lemma}
\begin{proof}
	The Koszul-dual operad $\Com^k$ admits a quadratic Gr\"obner basis with more or less any compatible ordering of shuffle monomials. Let us consider, for example, the reverse path-lexicographical ordering.  The only nontrivial operation in $\Com^k(2k+1)$ is given by the monomial $\nu^{!}(\nu^{!}(x_1,\ldots,x_{k+1}),x_{k+2},\ldots,x_{2k+1})$ that has to be the leading term of the Koszul dual quadratic Gr\"obner basis of $\Lie^k$.
\end{proof}

\begin{definition}
Let $\kgraph_d(n)$ denotes the subset of oriented graphs $\fGraphs^k_d(n)$ with $n$ white numbered vertices and with arbitrary amount of unordered black vertices such that each internal (black) vertex has no incoming edges and exactly $(k+1)$ outgoing edges. Multiple edges are not allowed.
	\end{definition}
Note that each edge in $\kgraph$ starts in an internal vertex and ends in an external one. Hence, the orientation of a graph is prescribed by its shape and can be omitted for simplicity.
Moreover,  there is an obvious bijection between monomials in the free (super)commutative algebra generated by the set of elements $\{\omega_{S}| S\subset [1n], |S|=k+1\}$ and the aforementioned set of graphs $\kgraph(n)$, such that each multiple $\omega_{S}$ produces an internal black vertex connected with the subset of external white vertices indexed by the elements of $S$.

\begin{lemma}
	\label{lem::algebra::span}	
	If the graph $\Gamma\in\kgraph_{d+1}(n)$ contains a loop then the corresponding monomial in $\calA^k_d(n)$ is equal to zero. Hence, the set of monomials assigned to the subset of trees $\ktree(n)\subset\kgraph(n)$ spans the algebra $\calA^k_d(n)$.
\end{lemma}
\begin{proof}
	Let us do the induction on the  length of a cycle.
	For the base of induction we notice that if a graph $\Gamma\in\kgraph(n)$ has a cycle of length $4$ (i.e. containing two internal and two external vertices) then the corresponding monomial is divisible by $\omega_{S} \omega_T$ with $|S\cap T|\geq 2$ and hence is zero.
	Suppose the given graph $\Gamma$ has a cycle of length $2m$. Let $\extv{a}$ be an external vertex in this cycle and let $\extv{b}$ and $\extv{c}$ be the adjacent external vertices in this loop.
	Then the corresponding monomial $\omega(\Gamma)$ has a factor $\omega_{a,b,I}\omega_{a,c,J}$ for appropriate subsets $I=\{i_1,\ldots,i_{k-1}\}$ and $J=\{j_1,\ldots,j_{k-1}\}$ of cardinality $k-1$.  
	If we apply the relation~\eqref{eq::OS::multi} to the following cyclically ordered set of indices:
	$$
	\{b,i_1,\ldots,i_{k-1},a,j_1,\ldots,j_{k-1},c\}
	$$
	we can rewrite
	$\omega_{a,b,I}\omega_{a,c,J}$ as a sum of quadratic monomials, such that each summand has a factor $\omega_{b,c,S}$ for appropriate $S$. Notice that the corresponding graphs will have cycles with the vertex $\extv{a}$ erased and $\extv{b}$ and $\extv{c}$ become external vertices that are adjacent in a given cycle. 
	Thus the length of the corresponding cycles becomes less and equal to $2m-2$.
\end{proof}

Let us extend the assignement $\varphi$ (given for quadratic terms in~\ref{eq::Lie->A:1},\ref{eq::Lie->A:2}) and define (recursively) the combinatorial map $\varphi$ that assings to each shuffle monomial  $\gamma\in \shufLie(S)$ 
a monomial $\varphi(\gamma)\in \calA^k_d(S)$. Note that $S$ is supposed to be a linearly ordered set that labels leaves of shuffle monomials:
\begin{enumerate}
	\item We pose $\varphi(\nu(x_{s_1},\ldots,x_{s_{k+1}})): = \omega_{\{s_1,\ldots,s_{k+1}\}}$;
	\item Suppose that  $v$ is a vertex of $\gamma \in \shufLie(S)$  whose all incoming edges are leaves marked by $s_1<\ldots<s_{k+1}$. (In other words $v$ corresponds to a single operation $\nu(x_{s_1},\ldots,x_{s_k})$.)
	Suppose that, moreover, the edge outgoing from $v$ is not the leftmost income of the corresponding inner vertex of $\gamma$. Then we set 
	$$
	\varphi(\gamma) := \omega_{\{s_1,\ldots,s_{k+1}\}} \varphi(\gamma^{v \leftrightarrow s_{1}})
	$$ 
	where $\gamma^{v \leftrightarrow s_{1}}$ is the shuffle monomial with the branch growing from $v$ replaced by a leaf indexed by a minimal leaf of this branch $s_1$.
	\item 
	Let $v$ be a vertex of $\gamma \in $ whose all incoming edges are leaves marked by $s_1<\ldots<s_{k+1}$ and  the edge outgoing from $v$ is the leftmost income of the corresponding inner vertex of $\gamma$. Then we set 
	$$
	\varphi(\gamma) := \omega_{\{s_1,\ldots,s_{k+1}\}} \varphi(\gamma^{v \leftrightarrow s_{k+1}})
	$$ 
	where $\gamma^{v \leftrightarrow s_{k+1}}\in \shufLie(S\setminus \{s_1,\ldots,s_k\})$ is the shuffle monomial with the branch growing from $v$ replaced by a leaf indexed by a maximal leaf of this branch $s_{k+1}$. But we pose a new order on the new set  of leaves $S':=S\setminus\{s_1,\ldots,s_{k}\}$ saying that $s_{k+1} > s$ (resp. $s_{k+1} < s$) iff $s_1 >s $ (resp. $s_{1} < s$). In other words we erase all elements $s_2,\ldots,s_{k+1}$ and replace $s_1$ with $s_{k+1}$.
\end{enumerate}
For example,
$$
\varphi\left( \nu(\nu(x_1,\nu(x_3,x_4,x_6),x_5),x_2,\nu(x_7,x_8,x_9)  \right) = \omega_{346} \omega_{789} \omega_{135} \omega_{527}.
$$
The set of admissible shuffle monomials in $\Pois_d^k$ represents the iterated $\mu$-multiplication of shuffle monomials in $\Lie^k$ depending on different subset of indices. Thus, we can easely assign to each shuffle monomial  $$\alpha=\mu(\gamma_1,\mu(\gamma_2,\ldots,\mu(\gamma_m,\gamma_{m+1})))\in \Com(m+1)\circ\Lie^k_d \subset  \Pois^k_d$$ the monomial $\varphi(\gamma_1)\ldots\varphi(\gamma_{m+1})\in\calA^k_d$.

\begin{lemma}
	The image under assignment $\varphi$ of admissible shuffle monomials of $\shufLie(nk+1)$
	constitute a basis of the $n$-th graded component of the algebra $\calA^k(nk+1)$ and the image of admissible shuffle monomials of  
	 $\Pois^{k}_d(S)$ constitute a basis of the algebra $\calA^k_d(S)$.
\end{lemma}
\begin{proof}	
Suppose that for a pair of $\shufLie(nk+1)$-monomials $\gamma,\gamma'$ we have the nonvanishing of the pairing $\langle\gamma, \varphi(\gamma')\rangle \neq 0$. Let us show by induction on $n$ that $\gamma=\gamma'$.
 Since we are dealing with trees one can find a multiple $\omega_{i_1,\ldots,i_{k+1}}$ of $\varphi(\gamma')$ that  has at most one index $i_s$ that appears in another multiple of  $\varphi(\gamma')$. Note, that $\varphi$ is constructed in the way that all external  vertices has at most two edges. That is $i_s$ appears as an index in exaclty one another generator. Now if $i_s$ is a minima of $I:=\{i_1,\ldots,i_{k+1}\}$ then there exists a unique vertex $v$ of $\gamma$ connected directly with leaves indexed by $i_1,\ldots,i_{k+1}$. Erasing the vertex $v$ from $\gamma$ and the corresponding multiple $\omega_{i_1,\ldots,i_{k+1}}$ from $\varphi(\gamma')$ we can proceed with induction. Similarly, if $i_s$ is not a minima of $I$, and this happens for all multiples of $\varphi(\gamma')$ then there exists a multiple $\omega_{J}$ yielding the property that there exists exaclty one index $j_s\in J$ that appears in other multiples of $\varphi(\gamma')$, and $j_s$ has to be the maxima of $J$ and the same induction procedure makes sence.
 
 Thus, we conclude that matrix pairing between $\shufLie(nk+1)$ and $\varphi(\shufLie(nk+1))$ is diagonal. Therefore, the matrix pairing between $\Pois^{k}_d(S)$ and their images under $\varphi$ is also given by diagonal matrix.
 On the other hand we know  that $\Pois^k_d(S)$ and $\calA_d^k(S)$ are dual vector spaces.
 Hence, both sets of elements are linearly independent and spans  $\Pois^k_d(S)$ and $\calA_d^k(S)$ respectively.
 \end{proof}

\begin{remark}
As mentioned in~\cite{Manin::Vallette} (see the sketch of the proof of 5.20)
one should expect  the direct implication of the relations~\eqref{eq::OS::cyclic},\eqref{eq::OS::multi} that the set of monomials $\varphi(\shufLie)$ spans the corresponding graded component but we are a bit lazy to go through the combersome combinatorics involved since we can use the Koszulness instead. 
\end{remark}

\section{Hopf derivations of $\Pois_d^k$ and graph complexes}
\label{sec::Def}
We do not recall here the precise definition of the deformation complex of a (Hopf) operad and refer to~\cite{Merkulov::deformation,Willwacher_grt,Fresse} for precise definition. The key point is that in order to define a biderivation (deformation) complex of the map of Hopf operads $\calP\stackrel{f}{\rightarrow}\calQ$ one has to find a fibrant resolution of $\calP$ (that is a quasi-free operad) and a cofibrant replacement of $\calQ$ (that has to be quasifree as an algebra).
Note that we do have found both fibrant and cofibrant replacement of the Hopf operad $\Pois^k_d$.
However, the full deformation complex is huge and the goal of this section is to repeat the arguments of~\cite{Willwacher_oriented} in order to compare the corresponding deformation complexes with the Kontsevich's graph complexes.
Since this is also out of the main strem of this note we do not recall the precise definitions of the graph complexes $\GC_d$, directed graph complexes $\dGC_d$ and oriented graph complexes $\GC_d^{\orient}$ and refer to the papers \cite{Willwacher_grt,Willwacher_oriented}.
The rough idea is that the graph complex consists of connected graphs that do not have external vertices and the differential is given by vertex splitting. As always the dual differential is given by edge contraction. The convention on grading is the same. Vertices has degree $d$ and edges has degree $1-d$. The ordinary graph complex $\GC_d$ consists of  graphs whose edges are not oriented. Respectively the directed graph complex $\dGC_d$ consists of graphs with a chosen orientation of each edge and the oriented graph complex $\GC_d^{\orient}$ is a subset of $\dGC_d$ containing graphs with no directed cycles. Finally, ${\hGCor_d}$ consists of directed graphs with \emph{hairs} and no directed cycles. Where by a \emph{hair} we mean an edge that starts in a vertex of a graph and has no end. The degree of a hair is set to be $0$. The number of loops defines a grading in all aforementioned graph complexes.  
Let ${\dGC_d}^{\odd}$ be the graded Lie subalgebra of the directed graph complex spanned by graphs whose number of loops is divisible by $k$. Note that as a subcomplex ${\dGC_d}^{\odd}$ is a direct summand of $\dGC_d$.
Respectively, let $\GC_d^{\orient,\odd}$ be the corresponding subcomplex of the oriented graph complex.
The standard action of the (oriented) graph complexes on the operad of (oriented) graphs(\cite{Willwacher_grt,Willwacher_oriented}) restricts to the natural action of $\GC_d^{\orient,\odd}$ on the operad $\Graphs^{\orient,\odd}_{d}$ by homotopy derivations as a Hopf operad.
\begin{theorem}
	\label{thm::Deformations}	
	The map 
	\[
	\GC_d^{\orient,\odd}\rightarrow \Def( \Omega({\pdu(\Pois_{d-1}^k)^{!}}) \rightarrow \fGraphs^{k\orient}_{d})[1] \simeq \BiDer^h(\Pois_{d-1}^k)
	\]
	is a quasiisomorphism of complexes. 	
\end{theorem}	
\begin{proof}
	The proof repeats the one presented in~\cite{Willwacher_oriented} and is conducted in several steps:
	
	Note that the Koszul resolution $\Omega({\pdu(\Pois_{d-1}^k)^{!}})\twoheadrightarrow \Pois^k_{d-1}$ contains, in particular, the dg-suboperad $\LL_\infty^k\{1-d\}\subset \Omega({\pdu(\Pois_{d-1}^k)^{!}})$ is generated by operations $\nu_{nk+1}$ that have the same nature as the one in $\ho\Pois^k_d$.
	First, one considers the filtration of the $\Def(\Pois_{d-1}^k\rightarrow \fGraphs^{k\orient}_{d})$ such that the associated graded differential is the operadic commutator with the generator $\mu_2$. The freeness of the algebra $\pdu\fGraphs^{k\orient}_{d}(n)$ implies that the first term of the corresponding spectral sequence can be identified with the graph subcomplex 	 ${\hGCorev_d}\subset \hGCor_d$ spanned by directed connected graphs $\Gamma$ with no oriented cycles, such that in addition the remainder modulo $k$ of  
	\begin{itemize}
		\setlength{\itemsep}{-0.4em}
		\item the number of outgoing edges in each external vertex of a graph $\Gamma$ and 
		\item the full amount of hairs (also called sometimes \emph{outgoing external legs}) 
	\end{itemize}
is equal to $1$. Consequently the loop order of $\Gamma\in {\hGCorev_d}$ is divisible by $k$.

	A graph $\Gamma\in\hGCorev_d$ with $nk+1$ hairs is considered as a homomorphism $\alpha_\Gamma:\ho\Pois_d^k(nk+1) \rightarrow \Graphs^{\orient,\odd}_{d}(nk+1)$ that maps the generator $\nu_{nk+1}$ to the odd graph $\tilde\Gamma$ whose inner part is isomorphic to the inner part of $\Gamma$ and hairs are connected with external vertices such that the external vertices become univalent. Of course, one has to sum up over these possibilities to make this map $\bbS_{nk+1}$-invariant. All other generators of the Koszul resolution of $\Pois_d^k$ are mapped to zero.

	Let us notice, that the differential in ${\hGCorev_d}$ is more complicated then the simple vertex splitting differential and can be defined pictorially in the following way:
	\begin{equation}\label{equ:hGCodifferential}
	\delta \Gamma = \sum_\nu \Gamma \bullet_\nu
	\begin{tikzpicture}[baseline=-.65ex]
	\node[int] (v) at (0,.3) {};
	\node[int] (w) at (0,-.3) {};
	\draw[-triangle 60] (v) edge (w);
	\end{tikzpicture}
	\pm
	\sum_{j\geq 0}
	\frac 1 {(kj+1)!}
	\begin{tikzpicture}[baseline=-.65ex]
	\node[int] (v) at (0,.3) {};
	\node[ext] (w) at (0,-.3) {$\Gamma$};
	\coordinate (w1) at (+.3,-.1);
	\coordinate (w2) at (-.3,-.1);
	\draw[-triangle 60] (v) edge (w) edge (w1) edge (w2);
	\end{tikzpicture}
	\pm
	\sum_{j\geq 0}
	\frac 1 {(kj+1)!}
	\begin{tikzpicture}[baseline=-.65ex]
	\node[ext] (v) at (0,.3) {$\Gamma$};
	\node[int] (w) at (0,-.3) {};
	\coordinate (w1) at (+.3,-.5);
	\coordinate (w2) at (-.3,-.5);
	\coordinate (w3) at (0,-.7);
	\draw[-triangle 60] (v) edge (w) (w) edge (w1) edge (w2) edge (w3);
	\end{tikzpicture}
	\end{equation}
	Here the first sum runs over vertices of $\Gamma$, and the symbol $\bullet_\nu$ shall mean that the graph on the right is inserted at vertex $\nu$. In the second term the black vertex has valence $kj+1$, and it is necessary to sum up over all possibilities to connect the edge pictorially ending in $\Gamma$ to vertices of $\Gamma$. In the last term one sums over all external legs of $\Gamma$ and connects one to the new vertex.

	Second, one defines a quasiisomorphism $\GCorev_d\rightarrow \hGCorev_d$ given by the following pictorial presentation:
	\begin{equation}\label{equ:hairymap}
	\Gamma \mapsto 
	\sum_{j=0}^\infty
	\frac 1 {(kj+1)!}
	\underbrace{
		\begin{tikzpicture}[baseline=-.65ex]
		\node (v) at (0,0) {$\Gamma$};
		\coordinate (v0) at (-.7,-1);
		\coordinate (v1) at (-.3,-1);
		\node (v2) at (.3,-1) {$\dots$};
		\coordinate (v3) at (.7,-1);
		\draw[-triangle 60] (v) edge (v0) edge (v1) edge (v3);  
		\end{tikzpicture}
	}_{kj+1\times}
	\end{equation}
	where the picture on the right means that one should sum over all ways of connecting $kj+1$ outgoing edges to the graph $\Gamma$ such that the resulting graph  belongs to $\hGCorev_d$. 
\end{proof}

\begin{corollary}
	\label{cor::Def::Twopois}
	The Lie algebra of homotopy derivations of the operad $\Pois_{d}^k$ is isomorphic to the Lie subalgebra $\GC_{d}^{\odd}\subset \GC_{d}$ spanned  by graphs with the loop order divisible by $k$.
\end{corollary}
\begin{proof}
	The equivalence of oriented graph complex and the nonoriented one of consecutive degrees was discovered in~\cite{Willwacher_oriented}:
	\begin{equation}
	\label{eq::GCor=GC}
	\GC_d^{\orient} \simeq \GC_{d-1} \oplus \text{ loops }
	\end{equation}
	and known to  preserve the loop order grading. The \emph{"loops"} elements has loop order $1$ and we have an isomorphism of subalgebras of graphs of loop order divisible by $k$:
	\[
	\GCorev_d \simeq \GC_{d-1}^{\odd}
	\]
\end{proof}

\appendix
\section{Combinatorial computations with graphs}
\subsection{Computing homology of $\fGraphs^{k\orient}$}
\label{sec::proof::Graphs}
In this section we prove Theorem~\ref{thm::H:Graph:1} which states that the operads of graphs $\fGraphs_{d+1}^{k}$ is equivalent to the operad $\Pois^k_d$.

\begin{proof}
It is easy to see that the combinatorial map $\psi:\ho\Pois_d^k \to \fGraphs_{d+1}^{k\orient}$ is a map of operads.	
Let us define  a natural surjective inverse map of complexes of $\mathbb{S}$-modules $\epsilon:\fGraphs_{d+1}^{k\orient} \to \ho \Pois_{d}^{k}$ which sends all graphs that contain a vertex with more than one incoming edge to zero and what remains is uniquely presented as an operadic tree. 
For example:
\[
\begin{tikzpicture}[scale=0.5]
\node[int] (v) at (-1,1.2) {};
\node[int] (w) at (1,1.2) {};
\node[ext] (v0) at (-3,0) {1};
\node[ext] (v1) at (-2,0) {2};
\node[ext] (v2) at (-1,0) {3};
\node[ext] (v3) at (0,0) {5};
\node[ext] (v4) at (1,0) {4};
\node[ext] (v5) at (2,0) {6};
\node[ext] (v6) at (3,0) {7};
\draw[-triangle 60] (v) edge (v1)  edge (v2) edge (v3);
\draw[-triangle 60] (w) edge (v)  edge (v4) edge (v5) edge (v6) edge[bend right=60] (v0);
\end{tikzpicture}
\quad
\stackrel{\epsilon}{\rightarrow} 
\quad
\nu_{5}(x_1,\nu_3(x_2,x_3,x_5),x_4,x_6,x_7)
\]
The map $\epsilon$ is not a map of operads, however, this is a one-sided inverse to the map of $\mathbb{S}$-modules ${\ho\Pois}_d^{k} \to \fGraphs_{d+1}^{k\orient}$ and in order to finish the proof of Theorem~\ref{thm::H:Graph:1} it is enough to show that the kernel of $\epsilon$ is an acyclic complex.
Let us define a combinatorial filtration on the  kernel of $\epsilon$, such that it is easy to show that the corresponding associated graded differential is acyclic.
\begin{definition}
	\label{def::bad:path}
	For each oriented graph $\Gamma\in \fGraphs_{d+1}^{k\orient}$ containing a vertex with more than one incoming edge we can assign a unique directed path (called  the \emph{ bad path}) yielding the following properties:
	\begin{itemize}
		\setlength\itemsep{-0.4em}	
		\item the bad path starts in a \emph{bad} vertex (= a vertex containing more than one incoming edge);
		\item the bad path ends in an external vertex and this is the lowest possible external vertex with a possible incoming path from a bad vertex;
		\item the bad path contains a unique bad vertex (the source of a path).
	\end{itemize}
\end{definition}
In the example below
we draw the bad path in red and fill the bad vertices in grey:
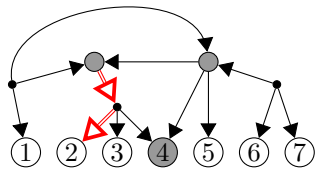
\[
\begin{tikzpicture}[scale=0.6]
\node[ext] (v0) at (0,0) {1};
\node[ext] (v1) at (1,0) {2};
\node[ext] (v2) at (2,0) {3};
\node[ext,fill=black!40] (v3) at (3,0) {4};
\node[ext] (v4) at (4,0) {5};
\node[ext] (v5) at (5,0) {6};
\node[ext] (v6) at (6,0) {7};
\node[int] (w0) at (-0.3,1.5) {};
\node[int] (w1) at (2,1) {};
\node[int,fill=black!40] (w2) at (1.5,2) {\tiny{$\phantom{\bullet}$}};
\node[int,fill=black!40] (w3) at (4,2) {\tiny{$\phantom{\bullet}$}};
\node[int] (w4) at (5.5,1.5) {};
\draw[-triangle 60] (w0) edge (v0)  edge (w2) edge[bend left=90] (w3);
\draw[-open triangle 60] (w1) edge[double,draw=red] (v1); 
\draw[-triangle 60] (w1) edge (v2) edge (v3);
\draw[-open triangle 60] (w2) edge[double,draw=red] (w1);
\draw[-triangle 60] (w3) edge (w2)  edge (v3) edge (v4);
\draw[-triangle 60] (w4) edge (w3)  edge (v5) edge (v6);
\end{tikzpicture}
\]
The vertex splitting differential may either increase the length of the bad path or increase the number of an external vertex in the target of a path.
Consider a filtration on $co\ker(\epsilon^{\dual})$ by the number of the bad external vertex and consider a grading by the length of the bad path on the associated graded complex.
The associated graded differential preserves the index of the bad external vertex and increases the length of the bad path. 
Moreover, the associated graded differential does not interact with edges that starts not at the bad path.
Consequently, the associated graded complex is a direct sum of complexes $\oplus_{G}C^{\udot}_{G}$. Where the summand $C^{\udot}_{G}$ is spanned by the set of graphs that coincide after contracting the bad path.
Let us show that each of these subcomplexes $C^{\udot}_{G}$ is acyclic.
Consider a particular graph $\Gamma\in C^{\udot}_{G}$.
Let $v$ be an vertex in the bad path of $\Gamma$ then $v$ has the following restriction on the set of incoming and outgoing edges outside of a bad path:
\begin{itemize}
	\setlength\itemsep{-0.4em}		
	\item if $v$ is not the source vertex of a bad path then $v$ has even positive number of outgoing edges (outside of a bad path) and no incoming edges;
	\item if $v$ is the starting vertex of a bad path then $v$ has even number of outgoing (outside of a bad path) edges and at least two incoming edges;
	\item if $v$ is a target external vertex in a bad path then $v$ has  zero number of outgoing edges.
\end{itemize}
Let us order the set $S(\Gamma)$ of edges starting at vertices in the bad path and ending outside of a bad path.
Let $A:=\QQ[e_1,\ldots,e_m]$ be the free commutative (polynomial) algebra generated by the variables indexed by $S(\Gamma)$. Let $A^{(k)}:=\oplus_{n\geq 0}A_{kn}\subset A$ be the $k$-th Segre power consisting of polynomials of degrees divisible by $k$.
The latter is known to be a quadratic Koszul algebra.
Consider the Bar resolution:
\[
B(A^{(k)}):= \bigoplus_{m\geq 0}  A^{(k)}\otimes \underbrace{A^{(k)}_{+}{\otimes}\ldots \otimes A^{(k)}_{+}}_{m}, \ d(a_0\otimes a_1\otimes \ldots \otimes a_m) = \sum_{i=0}^{m-1} (-1)^{i} a_0\otimes \ldots \otimes a_{i} a_{i+1} \otimes \ldots
\]
which is the resolution of a trivial left $A^{(k)}$-module.
We notice that the $\mathbb{S}_m$-invariants of the multilinear part of the bar resolution $B(A^{(k)})$ is dual to  the subcomplexes $C^{\udot}_G$ spanned by graphs which has the same shape as $\Gamma$ after contracting the bad path.
The multilinear part of $B(A^{(k)})$ is acyclic and thanks to Mashke's theorem the subspace of $\mathbb{S}_m$-invariants form an acyclic subcomplex and consequently, $C^{\udot}_G$ is also acyclic.

The case of an absence of outgoing edges outside of a bad path has to be considered separately, but in this case the length of a bad path may be either $0$ or $1$ and it is clear that it has no homology.
\end{proof}

\subsection{Computing homology of $\ICG^{k\orient}$}
\label{sec::proof::ICG}

This section is devoted to give a proof of the key Lemma~\ref{lem::ICG::trivalent}.
The proof is based on a collection of consecutive spectral sequence arguments such that the associated graded differential for each particular spectral sequence is a vertex splitting (edge contraction) that does not brake certain symmetries defined combinatorially in terms of graphs. 
Maschke's theorem is the key argument that helps in this type of computations.
One also should have in mind that any given pair of integers $n\geq 1$ and $l\geq 0$ the number of graphs in $\ICG^{k\orient}(n)$ with a given loop order $l$ is finite. Consequently all spectral sequences we are dealing with converge because all complexes are splitted to the direct sums of finite dimensional ones.

Let us start from certain combinatorial definition that leads to the generalization of Lemma~\ref{lem::ICG::trivalent}:
\begin{definition}
	\label{def::ICG_S}	
	For each subset $S\subset [1n]$ consider the subspace $\ICG_S^{k\orient}(n)\subset \ICG^{k\orient}(n)$ spanned by internally connected graphs $\Gamma$ yielding the two following properties:
	\begin{itemize}
		\item[($\imath$)] for all $s\in S$ the graph $\Gamma$ has a unique edge ending in the external vertex $\extv{s}$ and the source vertex of this edge has more than one outgoing edges:
		\(	\begin{tikzpicture}[scale=0.5]
		\node[int] (v) at (0,1) {};
		\node[ext] (v1) at (-1,0) {\small{$s$}};
		\node (v0) at (0,-0.2) {\small{$\ldots$}};
		\node (v2) at (1,0) {};
		\draw[-triangle 60] (v) edge (v1)  edge (v0) edge (v2);
		\end{tikzpicture}
		\)
		\item[($\imath\imath$)] we do not allow an internal vertex of valency $k+1$ such that with it has no incoming edges and $k$ (of $k+1$) outgoing edges  are connected to the $k$ different external vertices $s_1,\ldots,s_k\in S$:
		\begin{equation}
		\label{pic::bad::vertex}
		\begin{tikzpicture}[scale=0.5]
		\node[int] (v) at (0,1) {};
		\node[ext] (v1) at (-1,0) {\small{$s_1$}};
		\node[ext] (v0) at (0,-0.2) {\small{$s_2$}};
		\node[ext] (v3) at (3,0) {\small{$s_k$}};
		\node (v2) at (1.2,-0.2) {\small{\ldots}};
		\coordinate (v4) at (3,1);
		\draw[-triangle 60] (v) edge (v1)  edge (v0) edge (v3) edge (v4);
		\draw[red] (-2,1.5) -- (4,-0.5);
		\draw[red] (-2,-0.5) -- (4,1.5);
		\end{tikzpicture}.
		\end{equation} 
	\end{itemize}
\end{definition}
\begin{lemma}
	\label{lem::extS::trivalent}	
All non-trivial cohomology classes in $\ICG_S^{k\orient}(n)$ can be represented by linear combinations of internally trees whose all inner vertices either have one output and two inputs or have $k+1$ outgoing edges and no incoming edges. 

In particular, for each given loop degree the homology is concentrated in a unique homological degree.
\end{lemma}
Note that Lemma~\ref{lem::ICG::trivalent} is a particular case of Lemma~\ref{lem::extS::trivalent}  when $S=\emptyset$.
\begin{proof}
	The proof is the simultaneous (increasing) induction on the number $n$ of externally connected vertices, (decreasing) induction on the cardinality of the subset $S$ and (increasing) induction on the loop grading of a graph.
	
	For \emph{the base of induction} ($|S|=n$) we notice that there are no oriented graphs (with no oriented loops) in $\ICGS{[1n]}$. Indeed, each internal vertex of a graph $\Gamma\in \ICG^{k\orient}$ has at least one outgoing edge and if, in addition, each external vertex of $\Gamma$ is connected by a unique edge and the source of this edge has at least one edge ending in an internal vertex then one can construct a directed path of arbitrary length that avoids external vertices. However, the number of vertices is finite and hence this path contains a loop what is not allowed in $\ICG^{k\orient}$.
	
	\emph{(Induction step).}
	Suppose $m\in [1n]\setminus S$ and thanks to the induction on $n$ we may suppose that we are dealing with the ideal of graphs connected with $\extv{m}$. For each graph $\Gamma\in \ICGS{S}$ we can define the maximal subtree $T_m:=T_m(\Gamma)\subset \Gamma$ whose vertices are defined by the following property:
	
	\emph{ $w\in T_m$\ $\stackrel{def}{\Leftrightarrow}$\ There exists a unique (nondirected) path (with no selfintersections) that starts in $w$ and ends in $\extv{m}$.}
	
	Let $T_m^{\orient}$ be the subtree of $T_m$ spanned by those vertices $v$ of $T_m$ yielding the conditions: 
	\begin{itemize}
		\setlength{\itemsep}{0.4em}	
		\item[$(a1)$] the unique path starting at $v$ and ending in $\extv{m}$ is a directed path;
		\item[$(a2)$] if the vertex $v$ is internal (differs from $\extv{m}$) then it has a unique outgoing edge.
	\end{itemize}
	Let $\Gamma_1$,\ldots,$\Gamma_k$ be the set of internally connected components of the complementary graph $\Gamma\setminus T_m^{\orient}$ (internal vertices of $T_m^{\orient}$ are considered as external vertices of the complementary graph $\Gamma\setminus T_m^{\orient}$). Moreover, $\forall 1\leq i\leq k$ there exists exactly one vertex $v_i\in T_m^{\orient}$ such that there exists not less than one edge ending in $v_i$ that belongs to $\Gamma_i$. The uniqueness of $v_i$ is goverened by definition of $T_m\supset T_m^{\orient}$.  
	In Picture~\eqref{pic::tree::orient} we consider an example of a graph $\Gamma$.  The complementary graph $\Gamma\setminus T_m^{\orient}$ has three connected component drawn in different collors (green, yellow and blue).
	We draw the tree $T_m^{\orient}$ together with the incoming edges that correspond to connected components that is designed by corresponding colloring:
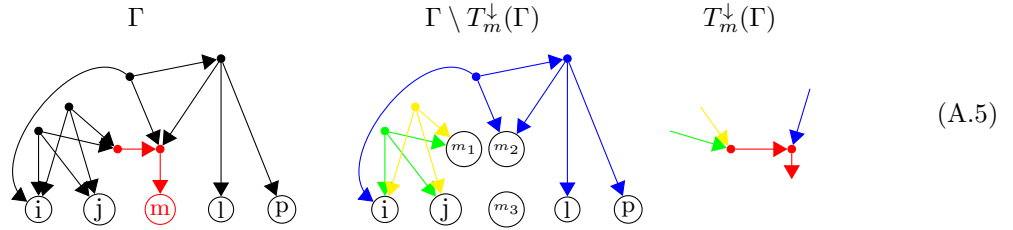
	\begin{equation}
	\label{pic::tree::orient}
	\begin{array}{ccc}
	\Gamma  & \Gamma\setminus T_m^{\orient}(\Gamma) & T_m^{\orient}(\Gamma)
	\\
	{
		\begin{tikzpicture}[scale=0.8]
		\node[int,red] (w1) at (0,1) {};
		\node[int,red] (w2) at (-0.7,1) {};
		\node[int] (w3) at (-2,1.3) {};
		\node[int] (w4) at (-0.5,2.2) {};
		\node[int] (w5) at (1,2.5) {}; 
		\node[int] (w6) at (-1.5,1.7) {};
		\node[ext] (v0) at (2,0) {\small{p}};
		\node[ext] (v1) at (-2,0) {i};
		\node[ext] (v2) at (-1,0) {j};
		\node[ext,red] (v3) at (0,0) {\small{m}};
		\node[ext] (v4) at (1,0) {\small{l}};
		\draw[-triangle 60] (w1) edge[red] (v3);
		\draw[-triangle 60] (w2) edge[red] (w1);
		\draw[-triangle 60] (w3) edge (w2);
		\draw[-triangle 60] (w3) edge (v1) edge (v2);
		\draw[-triangle 60] (w4) edge (w5);
		\draw[-triangle 60] (w4) edge[bend right=90] (v1);
		\draw[-triangle 60] (w4) edge (w1);
		\draw[-triangle 60] (w5) edge (w1);
		\draw[-triangle 60] (w5) edge (v4);
		\draw[-triangle 60] (w5) edge (v0);
		\draw[-triangle 60] (w6) edge (v1) edge (v2) edge (w2);
		\end{tikzpicture}
	}
	&{
		\begin{tikzpicture}[scale=0.8]
		\node[ext] (w1) at (0,1) {\tiny{$m_2$}};
		\node[ext] (w2) at (-0.7,1) {\tiny{$m_1$}};
		\node[int,green] (w3) at (-2,1.3) {};
		\node[int,blue] (w4) at (-0.5,2.2) {};
		\node[int,blue] (w5) at (1,2.5) {}; 
		\node[int,yellow] (w6) at (-1.5,1.7) {};
		\node[ext] (v0) at (2,0) {\small{p}};
		\node[ext] (v1) at (-2,0) {i};
		\node[ext] (v2) at (-1,0) {j};
		\node[ext] (v3) at (0,0) {\tiny{$m_3$}};
		\node[ext] (v4) at (1,0) {\small{l}};
		\draw[-triangle 60,green] (w3) edge (w2);
		\draw[-triangle 60,green] (w3) edge (v1) edge (v2);
		\draw[-triangle 60,blue] (w4) edge (w5);
		\draw[-triangle 60,blue] (w4) edge[bend right=90] (v1);
		\draw[-triangle 60,blue] (w4) edge (w1);
		\draw[-triangle 60,blue] (w5) edge (w1);
		\draw[-triangle 60,blue] (w5) edge (v4);
		\draw[-triangle 60,blue] (w5) edge (v0);
		\draw[-triangle 60,yellow] (w6) edge (v1) edge (v2) edge (w2);
		\end{tikzpicture}
	}
	&
	{
		\begin{tikzpicture}[scale=0.8]
		\node[int,red] (w1) at (0,1) {};
		\node[int,red] (w2) at (-1,1) {};
		\coordinate (w3) at (-2,1.3);
		\coordinate (w6) at (-1.5,1.7);
		\coordinate (w5) at (0.3,2); 
		\coordinate (v3) at (0,0.5);
		\draw (w1) edge[-triangle 60,red] (v3);
		\draw (w2) edge[-triangle 60,red] (w1);
		\draw (w6) edge[-triangle 60,yellow] (w2);
		\draw (w3) edge[-triangle 60,green] (w2);
		\draw (w5) edge[-triangle 60,blue] (w1);
		\end{tikzpicture}
	}
	\end{array}
	\end{equation}
	
	Consider the filtration $\calF^{p}$ of $\ICGS{S}$ by the number $p$ of internally connected components of the graph $\Gamma\setminus T_m^{\orient}(\Gamma)$. While ordering the corresponding graded components we ends up with the following isomorphism of symmetric collections:
	\[
\gr^{\calF}\ICG_{S}^{k\orient} \simeq \LL_{\infty} \circ \left(\calF^{1}/\calF^{>1}(\ICG_{S}^{k\orient})\right)	\]
Thanks to the Koszulness of the operad $\Lie$ it remains to explain that the quotient complex 	
$\ICGS{S}^1:=\calF^{1}/\calF^{>1}(\ICG_{S}^{k\orient})$ 	spanned by graphs $\Gamma\in\ICGS{S}$ with unique internally connected component $\Gamma_1$ of $\Gamma\setminus T_m^{\orient}(\Gamma)$ has appropriate cohomology.
	
	Note that the graph $\Gamma\setminus T_m^{\orient}(\Gamma)$ may have a unique internally connected component if and only if $T_m^{\orient}(\Gamma)$ either is equal to $\extv{m}$ or consists of one edge $\oneedge{v}{m}$.  
	Therefore, the complex $\ICGS{S}^1$ admits a decomposition $\ICGS{S}^1 = \ICGS{S}^1_{0} \oplus \ICGS{S}^1_{1}$ where the additional rightmost lower index corresponds to the number of edges in $T_m^{\orient}$. 
	Consider the homotopy $h:\ICGS{S}^1_{1}\twoheadrightarrow \ICGS{S}^1_{0}$ to the  first differential in the corresponding spectral sequence given by contraction of the edge $\oneedge{v}{m}$ if allowed. 
	The kernel of the surjective homotopy $h$ is spanned by graphs having more than one outgoing edge from the unique internal vertex $v\in T_m^{\orient}$. In particular, the complex $\ICGS{{S\sqcup\{m\}}}$ is a direct summand of the kernel of $h$ and can be reached by the decreasing induction on the cardinality of $S$.
	
	Thus, it remains to show that the complement of $\ICGS{{S\sqcup\{m\}}}$ in the kernel of $h$  (which we denote by $K$) is acyclic.
	Note that the complement $K$ consists of graphs that has the forbidden vertex~\eqref{pic::bad::vertex} connected with certain external vertices $t_1,\ldots,t_{k-1}\in S$ and $m$. Consider the filtration by the lexicographical order of the subset $T$ of first $k-1$ external vertices adjacent to $m$ (via exactly one internal vertex $v$). 
	 The associated graded complex $K_T$ admits additional twostep filtration $K_T=K_T^{k+1}\oplus K_T^{>k+1}$, where $K_T^{k+1}$ is spanned by graphs with the vertex $v$ of valency $k+1$. Note that in the latter case the vertex $v$ has only outgoing edges connected with $T$,$m$ and the remaining part of the graph.
	Consider the homotopy $h'$ to the associated graded differential given by contraction of the unique edge connecting $v$ and the remaining part of the graph:
	\begin{align*}
	h': \quad & K_{T}^{k+1} \phantom{.....................} \rightarrow \phantom{...................}  K_{T}^{>k+1} \\
	 \quad &
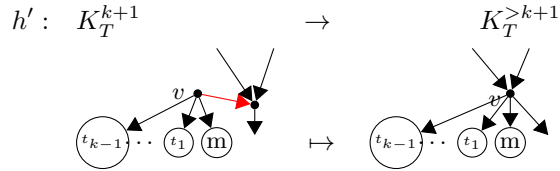
	\begin{tikzpicture}[scale =0.5]
	\node (w) at (-1,1.3) {\small{$v$}};
	\node[int] (w1) at (-0.5,1.3) {};
	\node[int] (w2) at (1,1) {};
	\node[ext] (v0) at (-3,0) {\tiny{$t_{k-1}$}};
	\node (v) at (-2,0) {\small{\ldots}};
	\node[ext] (v1) at (-1,0) {\tiny{$t_{1}$}};
	\node[ext] (v2) at (0,0) {\small{m}};
	\coordinate (u1) at (0,2.5);
	\coordinate (u2) at (1.5,2.5);
	\coordinate (u3) at (1,0.2);
	\draw[-triangle 60] (w1) edge (v0) edge (v1) edge (v2) edge[red] (w2);
	\draw (w2) edge[-triangle 60] (u3);
	\draw (u1) edge[-triangle 60] (w2);
	\draw (u2) edge[-triangle 60] (w2);
	\end{tikzpicture} 
	\quad \mapsto \quad 
	\begin{tikzpicture}[scale =0.5]
	\node (w) at (-0.4,1.1) {\small{$v$}};
	\node[int] (w1) at (0,1.3) {};
	\node[ext] (v0) at (-3,0) {\tiny{$t_{k-1}$}};
\node (v) at (-2,0) {\small{\ldots}};
\node[ext] (v1) at (-1,0) {\tiny{$t_{1}$}};
	\node[ext] (v2) at (0,0) {\small{m}};
	\coordinate (u1) at (-1,2.5);
	\coordinate (u2) at (0.5,2.5);
	\coordinate (u3) at (1,0.2);
	\draw[-triangle 60] (w1) edge (v0) edge (v1) edge (v2);
	\draw (w1) edge[-triangle 60] (u3);
	\draw (u1) edge[-triangle 60] (w1);
	\draw (u2) edge[-triangle 60] (w1);
	\end{tikzpicture}
	\end{align*}
	The homotopy $h'$ defines a bijection between graphs spanning $K_T^{k+1}$ and $K_T^{>k+1}$ except one particular case when the remaining part of the graph consists of one external vertex $\extv{t}$ and if $S=T$ or $S=T\sqcup\{t\}$ the cohomology of $K_T$ is spanned by the image of the generator $t_{t_1,\ldots,t_{k-1},t,m}$ with respect to the map $\psi$ defined  in~\eqref{eq::HICG::generators}.
As promised we finished with the increasing induction on the cardinality of the subset $S$ of external vertices. \end{proof}

\end{document}